\documentclass[11pt]{amsart}
\usepackage{amsthm}
\usepackage{amssymb}
\usepackage{MnSymbol}
\usepackage{graphicx}
\usepackage{subfigure}
\usepackage{geometry} 
\usepackage{psfrag}
\usepackage[all]{xy}
\title{The quivers of the hereditary algebras of type $\widetilde{A}_n$} 
\author{K. Baur, T. Br\"ustle}

\begin{document}

\newtheorem{lm}{Lemma}[section]
\newtheorem*{prop}{Proposition}
\newtheorem{satz}[lm]{Satz}

\newtheorem*{corollary}{Corollary}
\newtheorem{cor}[lm]{Korollar}
\newtheorem{theorem}[lm]{Theorem}
\newtheorem{thm}{Theorem}

\theoremstyle{definition}
\newtheorem{definition}[lm]{Definition}
\newtheorem*{defn}{Definition}
\newtheorem*{ue}{\"Ubung}
\newtheorem{bsp}{Beispiel}
\newtheorem*{ex}{Example}
\newtheorem*{exas}{Beispiele}
\newtheorem*{eigen}{Remark}
\newtheorem*{rem}{Remark}
\newtheorem{remark}[lm]{Remark}

\theoremstyle{remark}


\newcommand{\perm}{\operatorname{Perm}\nolimits}
\newcommand{\ZZ}{\operatorname{\mathbb{Z}}\nolimits}
\newcommand{\NN}{\operatorname{\mathbb{N}}\nolimits}
\newcommand{\PP}{\operatorname{\mathbb{P}}\nolimits}
\newcommand{\ad}{\operatorname{ad}\nolimits}
\newcommand{\im}{\operatorname{im}\nolimits}
\newcommand{\Hom}{\operatorname{Hom}\nolimits}
\newcommand{\id}{\operatorname{id}\nolimits}

\newcommand{\rep}{\operatorname{rep}\nolimits}
\newcommand{\Char}{\operatorname{char}\nolimits}

\begin{abstract}
This is a translation of the diploma thesis of Thomas Br\"ustle,\\
{\em Darstellungsk\"ocher der erblichen Algebren vom Typ $\widetilde{A}_n$}, 
written in German, at the University of Zurich, August 1990. \\ 
Supervisor: P. Gabriel. \\
Translated by K. Baur, May 2013 \\
\end{abstract}

\maketitle

\tableofcontents

\newpage

%
\section*{Introduction}
%
A quiver of type $\widetilde{A}_n$ has the following form \\
\psfragscanon
\psfrag{0}{\tiny $_{0}$}
\psfrag{1}{\tiny $_{1}$}
\psfrag{n}{\tiny $_{n}$}
\psfrag{n-1}{\tiny $_{n-1}$}
\includegraphics[scale=.6]{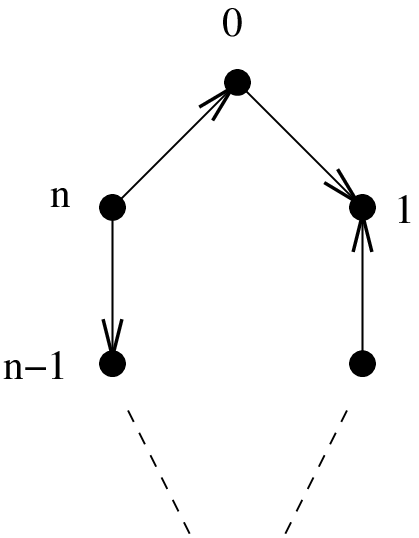}
with arbitrary orientation of the arrows. \\ 

The classification of (finite-dimensional) representations of quivers 
of this type is known. The goal of this thesis is, to describe the indecomposable representations and their morphisms 
by giving a quiver with relations. For this, we will restrict to the case of the quiver $K$ of the following form. 

$$
\begin{array}{cl}
\psfragscanon
\psfrag{K}{$K:$}
\psfrag{ag}{$_{\alpha_{g}}$}
\psfrag{an1}{$_{\alpha_{n-1}}$}
\psfrag{an}{$_{\alpha_{n}}$}
\psfrag{b0}{$_{\beta_0}$}
\psfrag{b1}{$_{\beta_1}$}
\psfrag{bg1}{$_{\beta_{g-1}}$}
\psfrag{0}{\tiny $_{0}$}
\psfrag{1}{\tiny $_{1}$}
\psfrag{2}{\tiny $_{2}$}
\psfrag{g-1}{\tiny $_{g-1}$}
\psfrag{g}{\tiny $_{g}$}
\psfrag{g+1}{\tiny $_{g+1}$}
\psfrag{n}{\tiny $_{n}$}
\includegraphics[scale=.5]{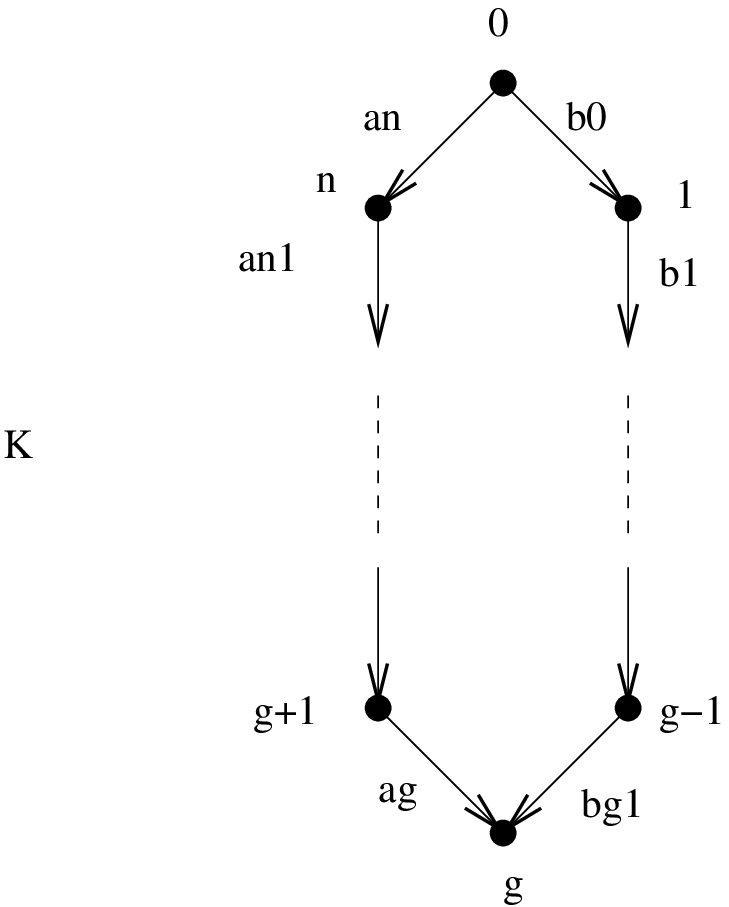} 
& 
\mbox{for some $1\le g\le n$}. \\ 
 &  \mbox{$g$ is the number of arrow in clockwise orientation,} \\ 
 & \mbox{$h:=n+1-g$ is the number of anti clockwise arrows.}
\end{array}
$$

\vskip 5pt

The running example throughout the thesis will be the case $g=3$, $n=4$: 
\begin{center}
\psfragscanon
\psfrag{K}{$K:$}
\psfrag{a4}{\tiny $\alpha_4$}
\psfrag{a3}{\tiny $\alpha_3$}
\psfrag{b0}{\tiny $\beta_0$}
\psfrag{b1}{\tiny $\beta_1$}
\psfrag{b2}{\tiny $\beta_2$}
\psfrag{0}{\tiny $_{0}$}
\psfrag{1}{\tiny $_{1}$}
\psfrag{2}{\tiny $_{2}$}
\psfrag{3}{\tiny $_{3}$}
\psfrag{4}{\tiny $_{4}$}
\includegraphics[scale=.6]{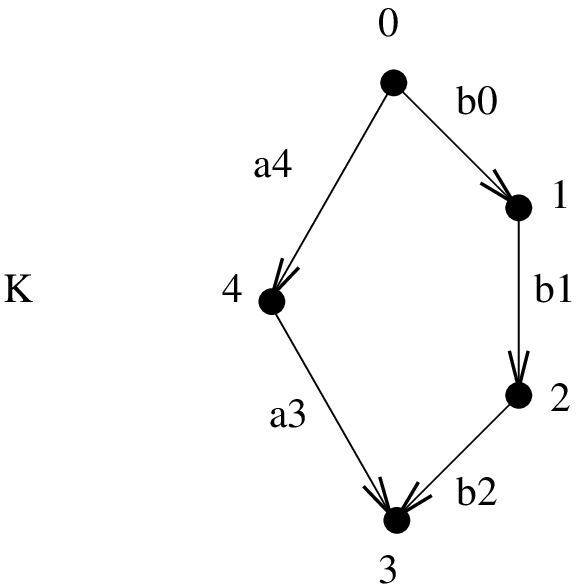}
\end{center}
%
\section{Classification of the representations}
%
We fix an algebraically closed field $k$ (commutative). 

A {\em representation $V$} of $K$ is a collection of $k$-vector spaces 
$V(x)$ for $x$ a vertex of $K$ and of $k$-linear maps $V(\gamma):V(x)\to V(y)$ for 
every arrow $\gamma:x\to y$ in $K$. 

A {\em morphism $\mu$:} $V\to W$ between two representations $V$ and $W$ of $K$ is a 
collection of $k$-linear maps $\mu(x):V(x)\to W(x)$ for $0\le x\le n$ such that we have 
$W(\gamma)\mu(x)=\mu(y)V(\gamma)$ for every arrow $\gamma:x\to y$ in $K$. Composition 
of morphisms is defined as composition of the linear maps in every vertex. 

The finite dimensional representations of $K$ form a $k$-category, we write $\rep K$ for it. 
For $V$, $W\in \rep K$ let $\Hom(V,W)$ be the space of morphisms between $V$ and $W$. 

To classify the representations we need another notion: 

A {\em (clockwise) walk}\footnote{Wanderweg in the original.} in $K$ is a pair $(p,q)\in \ZZ\times \ZZ$ with $p\le q$. To any 
walk $w$ we associate a representation $V_w$ of $K$ as follows: Assume first $0\le p\le n$ and 
set $\overline{i}=i + (n+1)\ZZ$ in 
$\ZZ/(n+1)\ZZ$, let $E_w$ be a vector space with basis $e_p,e_{p+1},\dots, e_q$. Then we define 
$$
V_w(x):=\bigoplus_{\overline{i}=\overline{x}}ke_i\subset E_w\quad  \mbox{for $0\le x\le n$} Ê
$$
and 
$$
 V_w(\beta_x):=\bigoplus_{\overline{i}=\overline{x}}ke_i\to\bigoplus_{\overline{j}=\overline{x+1}}ke_j \quad \mbox{for $0\le x\le g-1$}
$$ 
by setting 
$$ 
V_w(\beta_x)(e_i)=\left\{ \begin{array}{ll} e_{i+1} & \mbox{for $i<q$} \\  0 & \mbox{if $i=q$} \end{array}\right. 
$$
and
$$
 V_w(\alpha_x):=\bigoplus_{\overline{i}=\overline{x+1}}ke_i\to\bigoplus_{\overline{j}=\overline{x}}ke_j \quad \mbox{for $g\le x\le n$}
$$
through 
$$ 
V_w(\alpha_x)(e_i)=\left\{ \begin{array}{ll} e_{i-1} & \mbox{for $i>p$} \\  0 & \mbox{if $i=p$} \end{array}\right. 
$$
For arbitrary $p\in \ZZ$ we set 
$V_{(p,q)}:=V_{(p',q')}$ for $p' \in\{0,\dots,n\}$ with $p'\equiv p$ $\mod n+1$ and $q'=q+p'-p$. 

\begin{ex}
(Recall that $g=3$, $n=4$). For $w=(4,10)$, $V_w$ looks as follows: 
\begin{center}
\psfragscanon
\psfrag{+}{$\oplus$}
\psfrag{k4}{\small$ke_4$}
\psfrag{k5}{\small $ke_5$}
\psfrag{k6}{\small $ke_6$}
\psfrag{k7}{\small $ke_7$}
\psfrag{k8}{\small $ke_8$}
\psfrag{k9}{\small $ke_9$}
\psfrag{k10}{\small $ke_{10}$}
\psfrag{0}{\tiny 0}
\psfrag{1}{\tiny 1}
\psfrag{2}{\tiny 2}
\psfrag{3}{\tiny 3}
\psfrag{4}{\tiny 4}
\includegraphics[scale=.7]{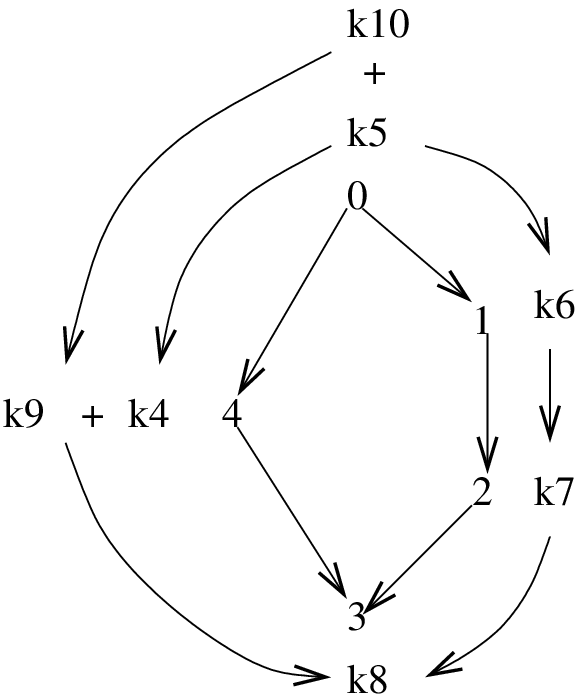}
\end{center}
\end{ex}

In addition to these, we consider representations $V_d^{\lambda}$ for any $\lambda\in k\cup\{\infty\}$, 
for any $d\in\NN\setminus\{0\}$, given as follows: 
$$
V_d^{\lambda}(x)=k^d,\ 0\le x\le n, \ V_d^{\lambda}(\gamma)=
             \left\{\begin{array}{ll} \lambda \id_d + J_d & \mbox{if $\gamma=\beta_{g-1}$} \\ \id_d & \mbox{if $\gamma\ne\beta_{g-1}$}\end{array}\right.
$$
where $J_d$ is the following $d\times d$-matrix: 
$J_d= \begin{pmatrix} 0 &  & \\ 1 & \ddots \\  & \ddots &\ddots\\  & & 1 & 0 \end{pmatrix}  \in k^{d\times d}$. 

Then we have

\begin{prop}
The representations $V_{(p,q)}$ with $0\le p\le n$, $q\ge p$ and the $V_d^{\lambda}$ for $\lambda\in k\cup\{0\}$, $d\in\NN\setminus\{0\}$ form 
a complete and irredundant list of indecomposables. 
\end{prop}
For a proof we refer to~\cite{gr}. 

\vskip 5pt

We introduce some notation. If $p\equiv 1,\dots, g$ $\mod n+1$ and 
$q\equiv g,g+1,\dots, n$ $\mod n+1$ then we call $w=(p,q)$ and $V_w$ 
{\em post-projective}. $P$ denotes the category whose objects are the post-projectives $V_{(p,q)}$. 

If $p\equiv 0,g+1,g+2,\dots,n$ $\mod n+1$ and $q\equiv 0,1,\dots,g-1$
$\mod n+1$
we call $w=(p,q)$ {\em pre-injective}. The category whose 
objects are the pre-injectives $V_w$ is denoted by $I$. 

The remaining representations from the list above are called {\em regular}. We write 
$$
\begin{array}{ll}
R^0 & \mbox{for the category with the objects $V_{(p,q)}$ with} 
  \left\{  \begin{array}{l}p\equiv 1,\dots,g\  \mod n+1Ê\\  
          q\equiv 0,\dots,g-1\  \mod n+1 \end{array} \right. \\ 
                                                                                            & \\
R^{\infty} &   \mbox{for the category with the objects $V_{(p,q)}$ with} 
\left\{  \begin{array}{l}p\equiv 0,g+ 1,\dots,n\  \mod n+1Ê\\  
  q\equiv g,\dots, n\ \mod n+1  \end{array} \right. \\ 
                                                                                           & \\
R^{\lambda} &  \mbox{for the category with the objects $V_d^{\lambda}$ 
for $\lambda\in k\setminus\{0\}, \ d\in \NN\setminus\{0\}$} 
\end{array} 
$$

%
\section{The post-projectives}\label{sec:postpr}
%
We start the study of the post-projective objects by ``rolling up'' the quiver $K$ (reversing arrows) to a quiver $\widetilde{K}$: 

Let $\widetilde{K}$ be the quiver whose vertices are $\ZZ$, with arrows $i\stackrel{\alpha_i}{\to} i+1$ for all $i\in\ZZ$ with $i\equiv g,\dots,n$ 
$\mod n+1$ and arrows $i+1\stackrel{\beta_i}{\to}i$ for all $i\in\ZZ$ with $i\equiv 0,1,\dots,g-1$ $\mod n+1$. 
$\widetilde{K}$ determines a quiver $\NN\widetilde{K}$ whose vertices are the pairs $(r,s)\in \NN\times\ZZ$ and whose arrows are 
given by $(r,\gamma):(r,i)\to (r,j)$ as well as $(r,\gamma'):(r,j)\to(r+1,i)$ for any arrow $\gamma:i\to j$ of 
$\widetilde{K}$ and any $r\in\NN$. 
A picture for the quiver $\NN\widetilde{K}$ for the running example is in Figure~\ref{fig:K-tilde}. 

\begin{figure}[ht]
\begin{center}
\psfragscanon
\psfrag{00}{\tiny $_{(0,0)}$}
\psfrag{01}{\tiny $_{(0,1)}$}
\psfrag{02}{\tiny $_{(0,2)}$}
\psfrag{03}{\tiny $_{(0,3)}$}
\psfrag{04}{\tiny $_{(0,4)}$}
\psfrag{05}{\tiny $_{(0,5)}$}
\psfrag{0-1}{\tiny $_{(0,-1)}$}
\psfrag{0-2}{\tiny $_{(0,-2)}$}
\psfrag{12}{\tiny $_{(1,2)}$}
\psfrag{13}{\tiny $_{(1,3)}$}
\psfrag{14}{\tiny $_{(1,4)}$}
\psfrag{22}{\tiny $_{(2,2)}$}
\psfrag{23}{\tiny $_{(2,3)}$}
\psfrag{24}{\tiny $_{(2,4)}$}
\psfrag{0b0}{\tiny $_{(0,\beta_0)}$}
\psfrag{0b1}{\tiny $_{(0,\beta_1)}$}
\psfrag{1b0}{\tiny $_{(1,\beta_0)}$}
\psfrag{0b0'}{\tiny $_{(0,\beta_0')}$}
\psfrag{0a3}{\tiny $_{(0,\alpha_3)}$}
\psfrag{0a3'}{\tiny $_{(0,\alpha_3')}$}
\psfrag{0K}{\tiny $_{\{0\}\times \widetilde{K}}$}
\psfrag{1K}{\tiny $_{\{1\}\times \widetilde{K}}$}
\includegraphics[scale=.6]{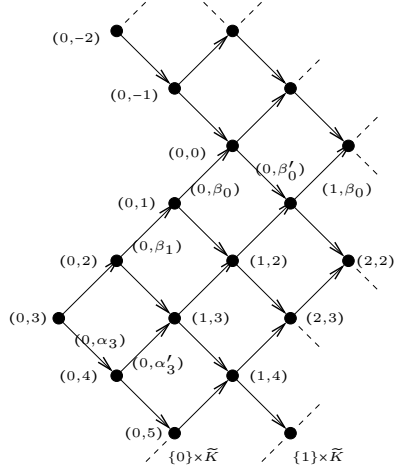}
\caption{$\widetilde{K}$ for the running example, $g=3$, $n=4$}
\label{fig:K-tilde}
\end{center}
\end{figure}

\vskip 5pt

We now define a functor $F:k\NN\widetilde{K}\to P$ 
(for any quiver $Q$ we use $kQ$ to denote the $k$-category of the paths of $Q$). 
We first give a map $F$ on the objects: 

To any vertex $(r,s)\in\NN\widetilde{K}$ there is exactly one path of maximal length in $\NN\widetilde{K}$ ending at $(r,s)$ that is a composition 
of arrows $\gamma:(l,i)\to (l',i+1)$. The starting point $(0,p)$ of this path determines $p\in\ZZ$ with 
$p\equiv 1,\dots,g$ $\mod n+1$. 

Analoguously, the path of maximal length ending at $(r,s)$ that is a composition of arrows 
$\gamma:(l,i)\to (l',i-1)$ starts at a point $(0,q)$ for some 
$q\in \ZZ$ with $q\equiv q,\dots,n$ 
$\mod n+1$. 
Using this, we set 
$$
F(r,s):=V_{(p,q)}\in P. 
$$
Facts: This map is surjective and the fibre of $(r,s)$ is the set of vertices $(r,s')$ with $s'\equiv s$ $\mod n+1$ 

The functor $F$ is determined uniquely, if we associate to any arrow $\gamma:X\to Y$ in $\NN\widetilde{K}$ a morphism 
$F\gamma:FX\to FY$. 
For this, we consider for any vertex $(r,s)$ the arrows $\gamma_1:(r,s)\to (r_1,s+1)$ and 
$\gamma_2:(r,s)\to (r_2,s-1)$. 
If $F(r,s)=V_{(p,q)}$ we have 
$$
\begin{array}{ll}
F(r_1,s+1)=& \left\{
                                  \begin{array}{ll} 
                                             V_{(p,q+1)}& \mbox{for $q\equiv g+1,\dots,n-1$ $\mod n+1$} \\   
                                             V_{(p,q+1+g)} & \mbox{for  $q\equiv n$ $\mod n+1$}
                                   \end{array} \right. \\
F(r_2,s-1)=& \left\{
                                  \begin{array}{ll} 
                                             V_{(p-1,q)}& \mbox{for $p\equiv 2,\dots, g$  $\mod n+1$} \\   
                                             V_{(p-1-h,q)} & \mbox{for  $p\equiv 1$ $\mod n+1$}
                                   \end{array} \right. 
\end{array}  
$$

Then we can define the morphisms $F\gamma_1$ and $F\gamma_2$: \\
$F\gamma_1:F(r,s)\to F(r_1,s+1)$ is the morphism sending $e_t$ to $e_t$ for $q\le t\le q$; \\
$F\gamma_2:F(r,s)\to F(r_2,s-1)$ is the morphism with $e_t\mapsto e_t$ for $p\le t\le q$ with $p\equiv 2,\dots, g$ 
$\mod n+1$, 
if $p\equiv 1$ $\mod n+1$, it sends $e_t$ to $e_{t+n+1}$. 

\begin{ex}
$F(1,1)\stackrel{F(1,\beta_1)}{\longleftarrow} F(1,2)\stackrel{F(0,\beta_1')}{\longleftarrow} F(0,1)$: 

\begin{center}
\psfragscanon
\psfrag{00}{\tiny $_{(0,0)}$}
\psfrag{01}{\tiny $_{(0,1)}$}
\psfrag{02}{\tiny $_{(0,2)}$}
\psfrag{03}{\tiny $_{(0,3)}$}
\psfrag{04}{\tiny $_{(0,4)}$}
\psfrag{0-1}{\tiny $_{(0,-1)}$}
\psfrag{0-2}{\tiny $_{(0,-2)}$}
\psfrag{12}{\tiny $_{(1,2)}$}
\psfrag{e1}{$_{e_1}$}
\psfrag{e2}{$_{e_2}$}
\psfrag{e3}{$_{e_3}$}
\psfrag{e4}{$_{e_4}$}
\psfrag{e5}{$_{e_5}$}
\psfrag{e6}{$_{e_6}$}
\psfrag{e7}{$_{e_7}$}
\psfrag{e8}{$_{e_8}$}
\psfrag{e9}{$_{e_9}$}
\psfrag{1b1}{\tiny $_{(1,\beta_1)}$}
\psfrag{0b1}{\tiny $_{(0,\beta_1)}$}
\psfrag{0b1'}{\tiny $_{(0,\beta_1')}$}
\includegraphics[scale=.7]{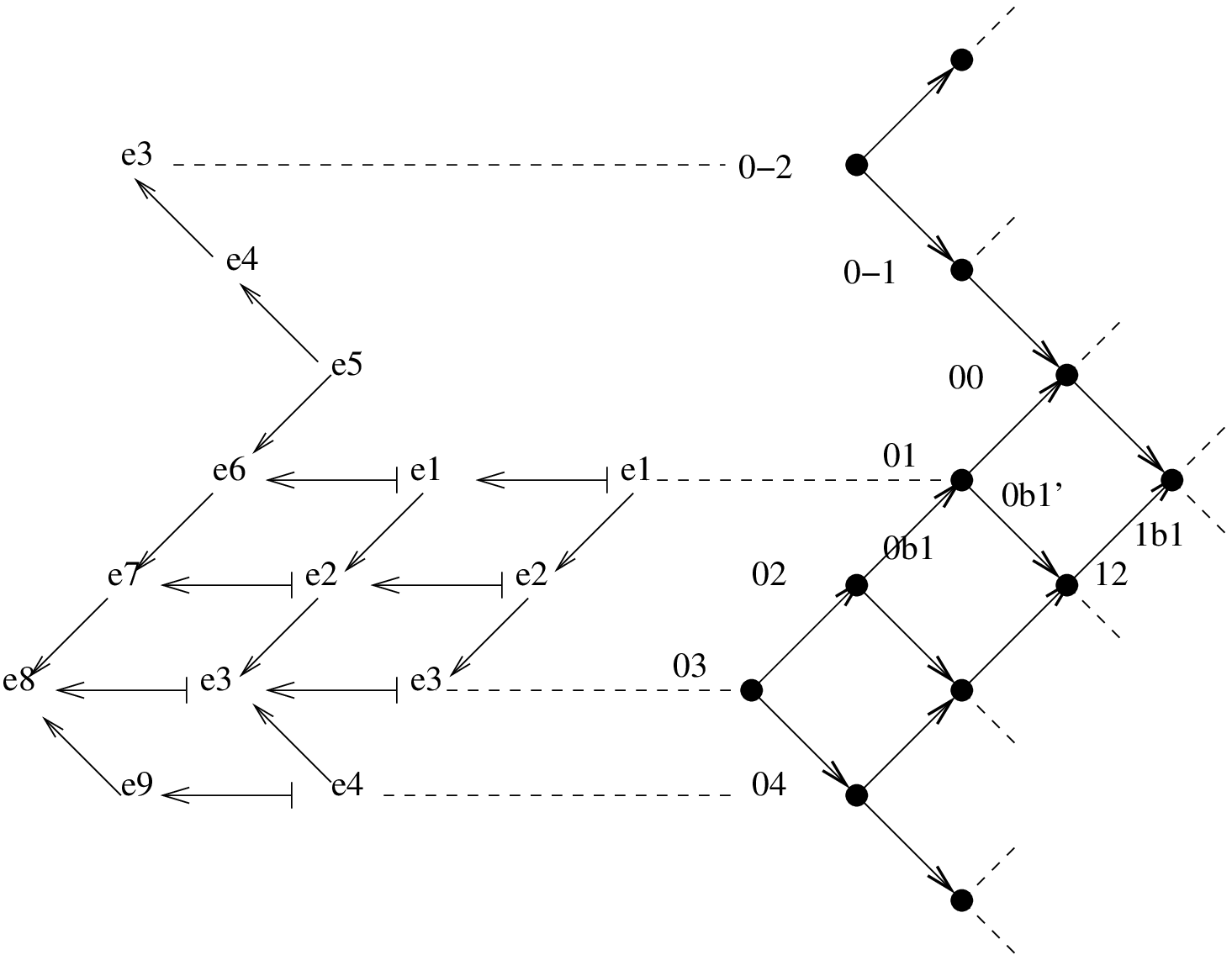}
\end{center}
\end{ex}

For all arrows $\gamma$, $\gamma'$, $\delta$, $\delta'$ of $\NN\widetilde{K}$ of the form 
$$
\xymatrix@R=+0.6pc @C=+0.1pc{
 & Y_1\ar[rd]^{\gamma'} & \\ 
X\ar[ru]^{\gamma}\ar[rd]_{\delta} & & Z \\
 & Y_2\ar[ru]_{\delta' }
}
\quad \mbox{we have $F\gamma'\circ F\gamma=F\delta'\circ F\delta$.}
$$
We will sometimes call this a {\em diamond\footnote{This definition is not in the diploma thesis. It is 
introduced to shorten the wording at 
times. 
A diamond will consist of four objects and four arrows, 
with an object $X$ at the start and an object $Z$ at the end.} starting at 
$X$ and ending at $Z$}, with arrows $X\stackrel{\tiny \gamma}{\to}Y_1\stackrel{\tiny \gamma'}{\to}Z$, 
$X\stackrel{\tiny \delta}{\to}Y_2\stackrel{\tiny \delta'}{\to}Z$.  

\vskip 5pt

Let $J$ be the ideal of $k\NN\widetilde{K}$ generated by the elements $\gamma'\gamma-\delta'\delta$ (whenever the four arrows $\gamma$, 
$\gamma'$, $\delta$, $\delta'$ are in a diamond as above). Then $J$ is annihilated by $F$, hence $F$ induces a functor $\overline{F}:k(\NN\widetilde{K})\to P$ 
where $k(\NN\widetilde{K}):=k\NN\widetilde{K}/J$. 

\begin{prop}
For any post-projective $V_{(p,q)}\in P$ and any vertex $Y$ of $\NN\widetilde{K}$ we have: \\
$\overline{F}: k(\NN\widetilde{K})\to P$ gives rise to an isomorphism 
$$
\begin{array}{lcl}
 \bigoplus_{\overline{F}X=V_{(p,q)}}\Hom(X,Y) &\stackrel{\sim}{\longrightarrow} & \Hom(V_{(p,q)},\overline{F}Y) \\ 
(\mu_X)_{\overline{F}X=V_{(p,q)}}& \longmapsto  &  \sum_{\overline{F}X=V_{(p,q)}}\overline{F}\mu_X
\end{array} 
$$
\end{prop}

The fibres of $F$ define an equivalence relation $\sim$ on $\NN\widetilde{K}_v$ and on $\NN\widetilde{K}_a$ (where we use 
$Q_v$ to denote the vertices of a quiver $Q$ and $Q_a$ to denote the arrows of $Q$). 

\vskip 5pt

We then let $Q^P$ be the quiver with $Q_v^P=\NN\widetilde{K}_v/\sim$ and $Q^P_a:=\NN\widetilde{K}_a/\sim$, denoting the 
equivalence class of $(r,s)\in\NN\widetilde{K}_v$ by $(r,s)_P$ and the equivalence class of $\gamma\in\NN\widetilde{K}_a$ by 
$\gamma_P$. 
We can visualize $Q^P$ as a tube. In our running example: 
$Q^P$ (page 8 of the thesis: note that in the thesis, the subscripts $P$ are not all there). 

\psfragscanon
\psfrag{00}{\tiny $_{(0,0)_P}$}
\psfrag{01}{\tiny $_{(0,1)_P}$}
\psfrag{02}{\tiny $_{(0,2)_P}$}
\psfrag{03}{\tiny $_{(0,3)_P}$}
\psfrag{04}{\tiny $_{(0,4)_P}$}
\psfrag{13}{\tiny $_{(1,3)_P}$}
\psfrag{23}{\tiny $_{(2,3)_P}$}
\psfrag{14}{\tiny $_{(1,4)_P}$}
\psfrag{10}{\tiny $_{(1,0)_P}$}
\includegraphics[scale=.6]{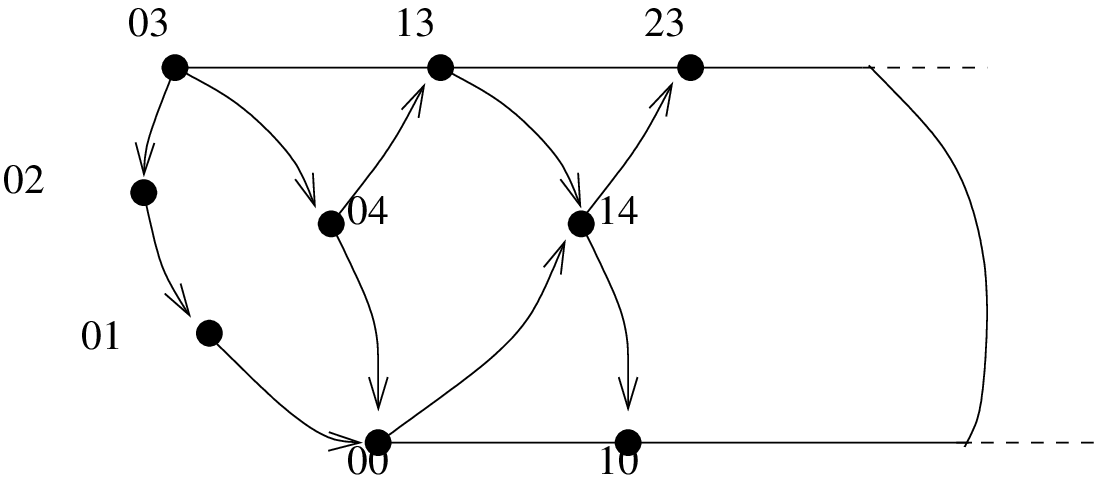}

It is more convenient to cut open $Q^P$ along the line through $(0,3)_P$ and $(1,3)_P$: 
\begin{figure}[ht]
\begin{center}
\psfragscanon
\psfrag{00}{\tiny $_{(0,0)}$}
\psfrag{01}{\tiny $_{(0,1)}$}
\psfrag{02}{\tiny $_{(0,2)}$}
\psfrag{03}{\tiny $_{(0,3)}$}
\psfrag{04}{\tiny $_{(0,4)}$}
\psfrag{13}{\tiny $_{(1,3)}$}
\psfrag{14}{\tiny $_{(1,4)}$}
\psfrag{ghm3}{\tiny $_{(ghm,3)}$}
\psfrag{ghm0}{\tiny $_{(ghm,0)}$}
\psfrag{0b0}{\tiny $_{(0,\beta_0)}$}
\psfrag{0a4}{\tiny $_{(0,\alpha_4)}$}
\psfrag{0a3}{\tiny $_{(0,\alpha_3)}$}
\psfrag{QmP}{\tiny $_{Q_m^P}$}
\psfrag{QP}{\tiny $_{Q^P}$}
\includegraphics[scale=.6]{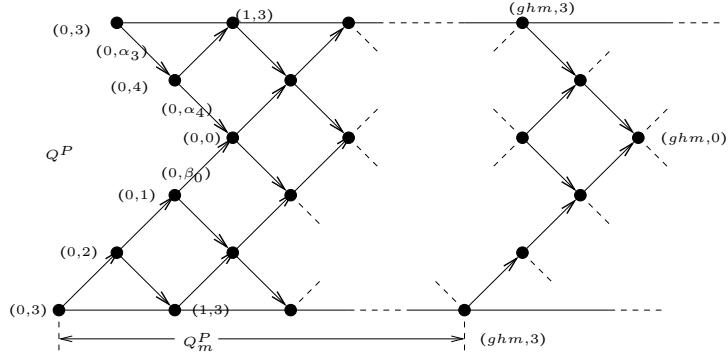}
\caption{$\widetilde{K}$ for the running example, $g=3$, $n=4$}
\label{fig:K-tilde-open}
\end{center}
\end{figure}

The functor $F:k\NN\widetilde{K}\to P$ determines, in a natural way, a functor $F_P:kQ^P\to P$ such that the map on the objectives 
is bijective. 

In the calcuations later we will restrict to certain subcategories $P_m$ of $P$. For this, we define for any $m\in \NN$ the full subquiver 
$Q_m^P$ of $Q^P$ on the vertices $(r,s)_P$ with $0\le s\le n$ and $0\le r\le ghm$. 

\begin{rem}
If we study $F_P:Q^P\to P$ more carefully, we can see that 
$$
F_P(r,0)_P=V_{(g-\mbox{MOD}(r,g),g+(n+1)\left[\mbox{DIV}(r,g) + \mbox{DIV}(r,h)+1\right] + \mbox{MOD}(r,h))}
$$
where MOD and DIV are the functions associating to any pair $(p,q)\in \ZZ\times \ZZ$ the uniquely defined numbers 
DIV$(p,q)$ and MOD$(p,q)$ with 
$$
p=\mbox{DIV}(p,q)\cdot q + \mbox{MOD}(p,q)\quad \mbox{and} \quad 0\le \mbox{MOD}(p,q)<q
$$
This expression becomes simpler if we restrict to numbers $r=ghm$ for $m\in \NN$: 
$$
F_P(ghm,0)_P=V_{(g,g+(n+1)(hm+gm+1))} = V_{(g,(n+1)(m(n+1)+1)+g)} 
$$
\end{rem}

We then define $P_m$ as the category whose objects are the $F_pX$ for all vertices $X$ of $Q_m^P$. To any 
$V_{(p,q)}\in P$ there is $m\in\NN$ such that $V_{(p,q)}\in P_m$. 

To make the description more transparent, we now assume $n\ge 2$. The Proposition above leads to the following: 
\begin{cor}
Let $k(Q_m^P)$ be the $k$-category defined by the quiver $Q_m^P$ with the relations $\gamma'\gamma=\delta'\delta$ 
for all arrows $\gamma$, $\gamma'$, $\delta$, $\delta'$ of $Q_m^P$ forming a diamond as above. Then $F_P$ induces 
an isomorphism 
$$
\overline{F}_P:k(Q_m^P)\to P_m
$$
\end{cor}

%
\section{The pre-injectives}\label{sec:pre-inj}
%
The following construction allows us to reduce the description of the pre-injective objects to the post-projective ones: 
We denote by $Q_m^I$ the quiver with a vertex $(r,s)_I$ for any vertex $(r,s)_P$ of $Q_m^P$ and an arrow 
$\gamma_I:(r_2,s_2)_I \to (r_1,s_1)_I$ for any arrow $\gamma_P:(r_1,s_1)_P\to (r_2,s_2)_P$ of $Q_m^P$. 

We view the transition from $Q_m^P$ to $Q_m^I$ as a reflection $S$: we set $S(r,s)_P:=(r,s)_I$ and 
$S(r,s)_I=(r,s)_P$ (for all vertices of $Q_m^P$ resp. of $Q_m^I$). In the example:

\begin{figure}[ht]
\begin{center}
\psfragscanon
\psfrag{00}{\tiny $_{(0,0)}$}
\psfrag{01}{\tiny $_{(0,1)}$}
\psfrag{03}{\tiny $_{(0,3)}$}
\psfrag{ghm3}{\tiny $_{(ghm,3)}$}
\psfrag{ghm0}{\tiny $_{(ghm,0)}$}
\psfrag{QP}{\tiny $_{Q^P}$}
\psfrag{QI}{\tiny $_{Q^I}$}
\includegraphics[scale=.45]{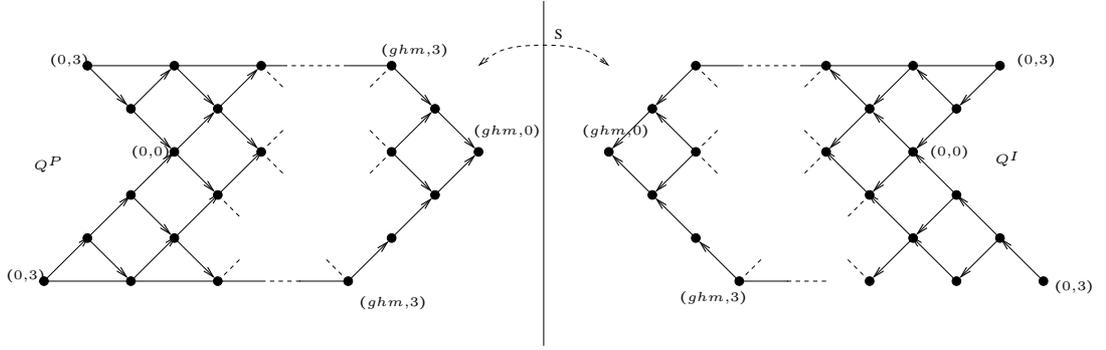}
\caption{Reflection $Q_m^P$ $\leftrightarrow$ $Q_m^I$ for the running example, $g=3$, $n=4$}
\label{fig:QP-QI}
\end{center}
\end{figure}

To get a functor $F_I:kQ_m^I\to I$, we use an equivalent construction on $\rep K$: 
Let $K^T$ be the quiver with vertices $K_v^T=K_v$ and arrows $\gamma^T:y\to x$ for every arrow $\gamma:x\to y$ in $K$. 
For every $V\in\rep K$ we then define $V^T\in \rep K^T$ using the dual space $V(x)^T$ in every vertex $x$ and the dual maps 
$V(\gamma)^T:V(y)^T\to V(x)^T$ for every arrow $\gamma^T:y\to x$ of $K^T$. 

By setting $\mu^T:=(\mu(0)^T,\dots,\mu(n)^T)\in\Hom(W^T,V^T)$ for any $\mu\in\Hom(V,W)$ we get a functor $?^T:\rep K\to \rep K^T$ 
with 
$\Hom(V,W)\stackrel{\sim}{\to}\Hom(W^T,V^T)$ for all $V,W\in\rep K$. 

By tilting (``kippen'') $K^T$ we can associate to any $V^T$ a representation $\overline{V}\in \rep K$; in the example: 

\begin{center}
\psfragscanon
\psfrag{V0}{\tiny $_{V(0)^T}$}
\psfrag{V1}{\tiny $_{V(1)^T}$}
\psfrag{V2}{\tiny $_{V(2)^T}$}
\psfrag{V3}{\tiny $_{V(3)^T}$}
\psfrag{V4}{\tiny $_{V(4)^T}$}
\psfrag{VT-V}{$V^T\mapsto \overline{V}$}
\includegraphics[scale=.6]{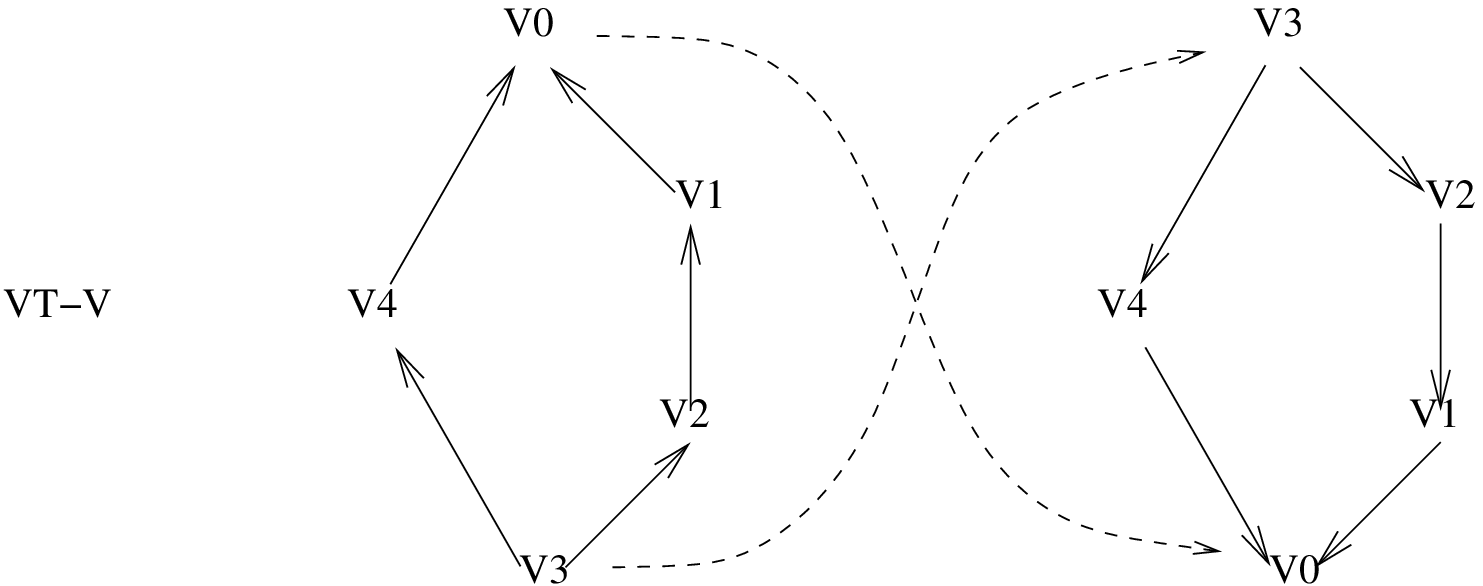}
\end{center}

In general, we denote by $G$ the permutation of $\{0,\dots, n\}$ with $Gx=gx$ for $x=0,\dots, g$ and $Gx=n+g+1-x$ for $x=g+1,\dots, n$. 
$G$ induces a bijection between the arrows of $K^T$ and of $K$, sending $\gamma^T:y\to x$ to the arrow $G\gamma:Gy\to Gx$ of $K$. 

We can then describe a functor $G:\rep K^T\to \rep K$ as follows: For $U\in \rep K^T$ let $GU\in \rep K$ be defined by 
$$
GU(x)=U(Gx) \quad x=0,\dots, n \quad GU(G\gamma)=U(\gamma^T) \quad \mbox{for any arrow $\gamma^T$ of $K^T$}; 
$$
to a morphism $\nu=(\nu(0),\dots,\nu(0))\in\Hom(U,U')$ we assign \\ $G\nu=(\nu(G0),\dots,\nu(Gn))\in \Hom(GU,GU')$. 

The composition $G\cdot ?^T=:\overline{?}$ is then a functor $\overline{?}:\rep K\to \rep K$ with the following property: 

For any $V,W\in \rep K$, the map $\Hom(V,W)\stackrel{\sim}{\to} \Hom(\overline{W},\overline{V})$, $\mu\mapsto \overline{\mu}$, 
is an isomorphims. 

We extend the permutation $G$ to a function $G:\ZZ\to \{0,\dots, n\}$ by setting, for all $z\in\ZZ$, $Gz:=G(\mbox{MOD}(z,n+1))$. 

\begin{prop}
Let $p,q\in \ZZ$ with $p\le q$, let $p'$, $q'\in\ZZ$ with $p'\equiv Gq$ $\mod n+1$, 
$q'\equiv Gp$ $\mod n+1$ and $q'-p'=q-p$. 
Then there is an isomorphism 
$$
\varphi:\overline{V}_{(p,q)}\stackrel{\sim}{\longrightarrow}V_{(p',q')} 
$$
\end{prop}

\begin{proof}
The definition ov $V_{(p,qP}$ is given with the help of a vector space $E_{(p,q)}$ with basis $e_p,\dots,e_q$. Let $\check{e}_p, \dots, \check{e}_q$
be the dual basis in $E_{(p,q)}^T$. Then we have \\
$\overline{V}_{(p,q)}(x)=\oplus_{\overline{i}=\overline{Gx}}k\check{e}_i, \quad 0\le x\le n$. 

One checks that the map $\varphi:\overline{V_{(p,q)}}\to V_{(p',q')}$, given by $\varphi(\check{e}_{p+t}=e_{q'-t}$, $t=0,\dots, p-q$, is an isomorphism. 
\end{proof}

\begin{ex}
We consider our running example, looking at 
$$
V_{(3,9)}\stackrel{?^T}{\longmapsto} V_{(3,9)}^T\stackrel{G}{\longmapsto} \overline{V}_{(3,9)}\stackrel{\varphi}{\longrightarrow} V_{(4,10)}: 
$$

\psfragscanon
\psfrag{e3}{\small$e_3$}
\psfrag{e4}{\small $e_4$}
\psfrag{e5}{\small $e_5$}
\psfrag{e6}{\small $e_6$}
\psfrag{e7}{\small $e_7$}
\psfrag{e8}{\small $e_8$}
\psfrag{e9}{\small $e_9$}
\psfrag{e10}{\small $e_{10}$}
\psfrag{f3}{\small$\check{e}_3$}
\psfrag{f4}{\small $\check{e}_4$}
\psfrag{f5}{ \small$\check{e}_5$}
\psfrag{f6}{ \small$\check{e}_6$}
\psfrag{f7}{\small $\check{e}_7$}
\psfrag{f8}{\small $\check{e}_8$}
\psfrag{f9}{\small $\check{e}_9$}
\psfrag{V39}{$V_{(3,9)}$}
\psfrag{V39T}{$V_{(3,9)}^T$}
\psfrag{V-}{$\overline{V_{(3,9)}}$}
\psfrag{V410}{$V_{(4,10)}$}
\psfrag{?T}{$?^T$}
\psfrag{G}{$G$}
\psfrag{phi}{\small$\varphi$}
\psfrag{iso}{\small$\sim$}
\includegraphics[scale=.65]{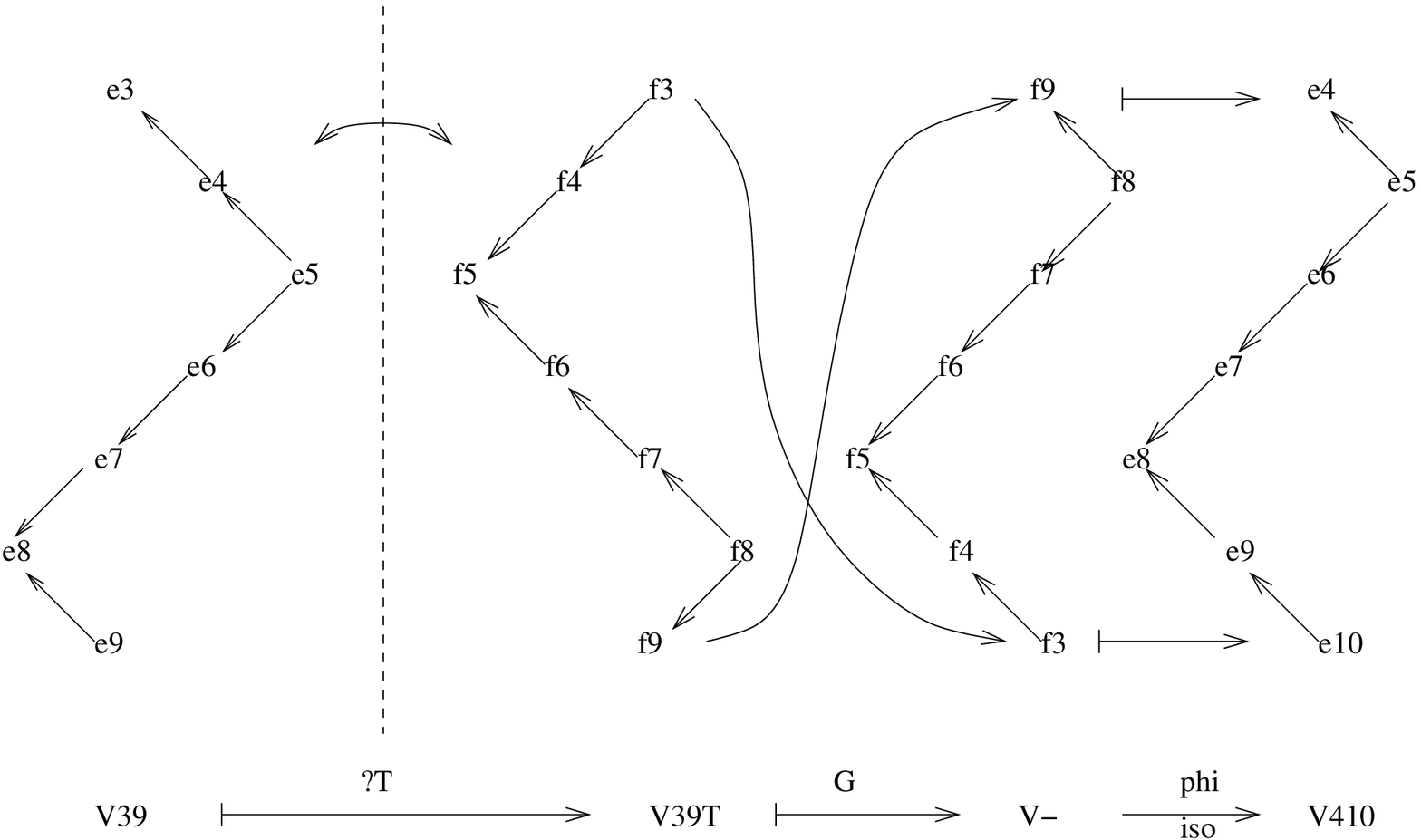}
\end{ex}

For any $X\in Q_m^P$ with $F_PX=V_{(p,q)}$ we define the isomorphism $\varphi:\overline{V_{(p,q)}}\stackrel{\sim}{\longrightarrow}V_{(p',q')}$ by $\varphi_X$. 

We are now ready to define a functor $F_I:kQ_m^I\to I$ by setting $F_IX=\varphi_{SX}(\overline{F_PSX})$ for any vertex $X$ of $Q_m^I$ and, for any arrow 
$\gamma:Y\to X$ of $Q_m^I$ we get a morhpism $F_I\gamma:F_IY\to F_IX$ by setting 
$F_I\gamma=\varphi_{SX}\circ \overline{FS\gamma}\circ\varphi^{-1}_{SY}$ with the arrow $S\gamma:SX\to SY$ of $Q_m^P$. 

\vskip 5pt 

(Let $n\ge 2$.)  
\begin{prop}
Let $k(Q_m^I)$ be the $k$-category defined by the quiver $Q_m^I$ and the relations $\gamma'\gamma=\delta'\delta$ for any diamond from $X$ to $Z$ 
in $Q_m^I$, let $I_m$ be the $k$-category with objects $F_I X$ for 
all $X\in Q_m^I$. Then $F_I$ induces an isomorphism 
$$
\overline{F}_I:k(Q_m^I)\to I_m. 
$$
\end{prop}

%
\section{The regular ones}
%

We now describe $R^0$. 
Let $Z$ be the quiver with vertices $(r,s)\in \NN\times \ZZ$ and arrows $\pi(r,s):(r+1,s+1)\to (r,s)$ and $\rho(r,s):(r,s)\to (r+1,s)$ for 
every vertex $(r,s)$. 

\begin{center}
\psfragscanon
\psfrag{Z}{$Z:$} 
\psfrag{22}{\tiny $_{(2,2)}$}
\psfrag{12}{\tiny $_{(1,2)}$}
\psfrag{01}{\tiny $_{(0,1)}$}
\psfrag{11}{\tiny $_{(1,1)}$}
\psfrag{00}{\tiny $_{(0,0)}$}
\psfrag{0-1}{\tiny $_{(0,-1)}$}
\psfrag{10}{\tiny $_{(1,0)}$}
\psfrag{20}{\tiny $_{(2,0)}$}
\psfrag{21}{\tiny $_{(2,1)}$}
\psfrag{31}{\tiny $_{(3,1)}$}
\psfrag{r11}{\tiny $_{\rho(1,1)}$}
\psfrag{p21}{\tiny $_{\pi(2,1)}$}
\psfrag{p11}{\tiny $_{\pi(1,1)}$}
\psfrag{r00}{\tiny $_{\rho(0,0)}$} 
\includegraphics[scale=.6]{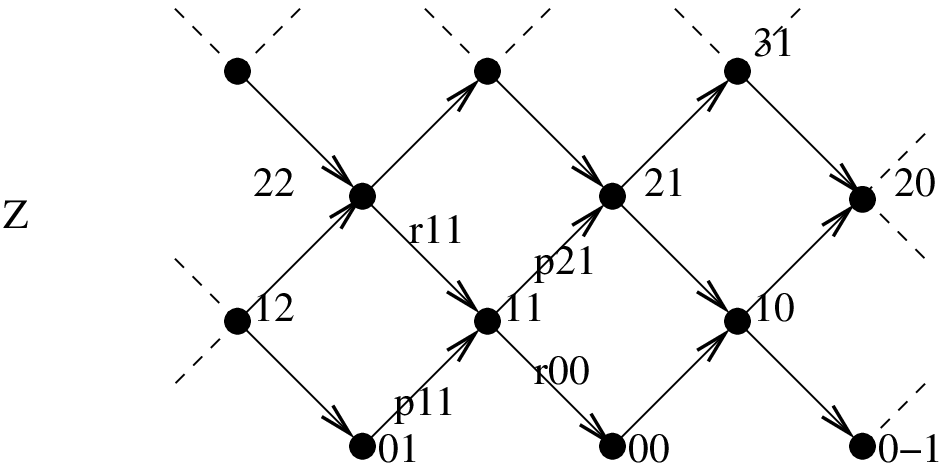}
\end{center}

We define a map $F$ which associates to any vertex $(r,s)$ an element $F(r,s)$ of $R^0$: let $(r,s)$ be a vertex of $Z$, define 
$p\in\{1,\dots,g\}$ by asking $p\equiv s-r$ $\mod g$ and $q'\in\{0,\dots, g-1\}$ by asking $q'\equiv s$ $\mod g$; 
let $t$ be the number of $s'\in\ZZ$ with $s'\equiv 0$ $\mod g$ and $s-r\le s'\le s$. 
\begin{center}
\psfragscanon
\psfrag{r}{\tiny $\rho$}
\psfrag{pi}{\tiny $\pi$}
\psfrag{q'}{\tiny $_{(s\equiv q')}$}
\psfrag{p}{\tiny $_{(s-r\equiv p)}$}
\psfrag{rs}{\tiny $_{(r,s)}$}
\psfrag{0s}{\tiny $_{(0,s)}$}
\psfrag{0s1}{\tiny $_{(0,s'_1)}$}
\psfrag{0s2}{\tiny $_{(0,s'_2)}$}
\psfrag{0st}{\tiny $_{(0,s'_t)}$}
\psfrag{0sr}{\tiny $_{(0,s-r)}$}
\includegraphics[scale=.5]{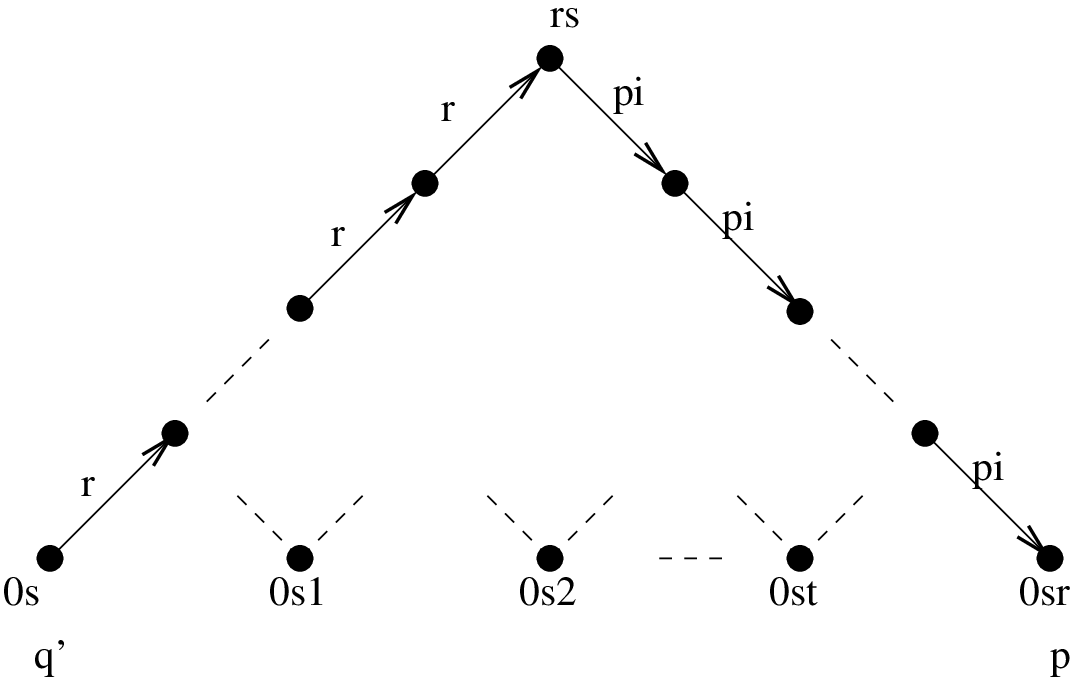}
\end{center}
Then we define $F_{(r,s)}:=V_{(p,q'+t(n+1))}$. 

Next we associate morphisms to the arrows $\pi(r,s)$ and $\rho(r,s)$ (for any vertex $(r,s)$). To describe this, 
assume $F(r,s)=V_{(p,q)}$. 

We let $F\pi(r,s)$ and $F\rho(r,s)$ be the following morphism: 
$$
\begin{array}{ll}
F\pi(r,s):F(r+1,s+1)\to F(r,s) & e_t\longmapsto \left\{ \begin{array}{ll}   e_t & \mbox{if } t=p,\dots, q \\ 0 & \mbox{if } t>q \end{array} \right. \\
 & \\ 
  F\rho(r,s):F(r,s)\to F(r+1,s)  & e_t\longmapsto \left\{ \begin{array}{ll}   e_t & \mbox{if } t=p,\dots, q \mbox{ and }p= 2,\dots,g 
     \\ e_{t+n+1} & \mbox{if } t=p,\dots, q t>q \mbox{ with $p=1$} \end{array} \right. 
\end{array}
$$

\begin{ex}
$F(2,3)\stackrel{F\pi(1,2)}{\longrightarrow} F(1,2)\stackrel{F\rho(1,2)}{\longrightarrow} F(2,2)$: 
\psfragscanon
\psfrag{e1}{$_{e_1}$}
\psfrag{e2}{$_{e_2}$}
\psfrag{e3}{$_{e_3}$}
\psfrag{e4}{$_{e_4}$}
\psfrag{e5}{$_{e_5}$}
\psfrag{e6}{$_{e_6}$}
\psfrag{e7}{$_{e_7}$}
\includegraphics[scale=.55]{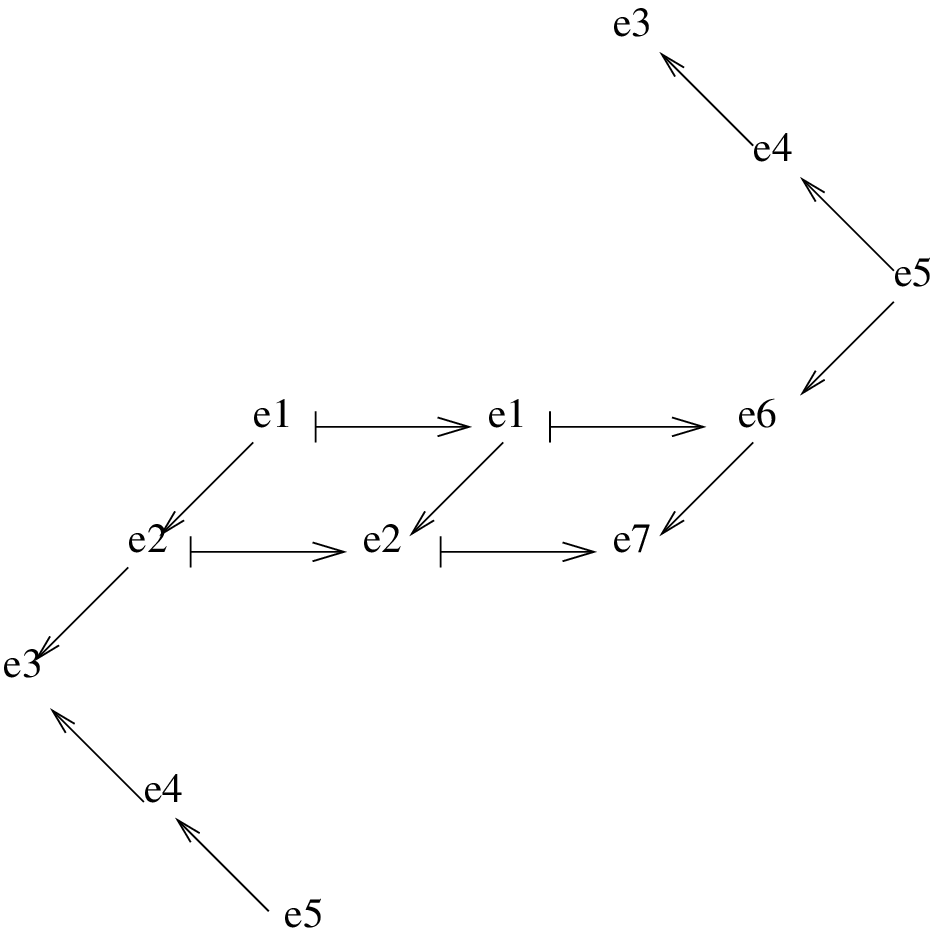}
\end{ex}

We have thus determined uniquely a functor $F:kZ\to R^0$. It is surjective on objectives and we have 
$$
\begin{array}{lcl}
F(r,s) & = & F(r,s+g) \\
F\pi(r,s) & = & F\pi(r,s+g) \\
F\rho(r,1) & = & F\rho(r,s+g) 
\end{array}
$$
for any vertex $(r,s)$. Analoguously to the construction for the post-projective objects, we define a quiver $Q^0:= Z/\sim$, 
where the equivalence relation $\sim$ is given by the fibres of $F$; 
we denote the equivalence class of $(r,s)$ by $(r,s)_0$, we write $\pi_0(r,s)$ and $\rho_0(r,s)$ for the equivalence 
classes of $\pi(r,s)$ and of $\rho(r,s)$ respectively. 

Let $Q_m^0$ be the full subquiver of $Q^0$ containing the vertices $(r,s)_0$ with 
$1\le s\le g$ and $0\le r\le gm+s$ (for 
$m\in \NN$). In the example, the shape of $Q_m^0$ for even $m$ is the following: 

\begin{center}
\psfragscanon
\psfrag{Q0}{$Q_m^0:$}
\psfrag{01}{\tiny $_{(0,1)}$}
\psfrag{02}{\tiny $_{(0,2)}$}
\psfrag{g03}{\tiny $_{(0,g)=(0,3)}$}
\psfrag{03g}{\tiny $_{(0,3)=(0,g)}$}
\psfrag{gmg}{\tiny $_{(gm,g)}$}
\psfrag{gm1g}{\tiny $_{(gm+1,g)}$}
\psfrag{gm2g}{\tiny $_{(gm+2,g)}$}
\psfrag{gmgg}{\tiny $_{(gm+g,g)}$}
\includegraphics[scale=.5]{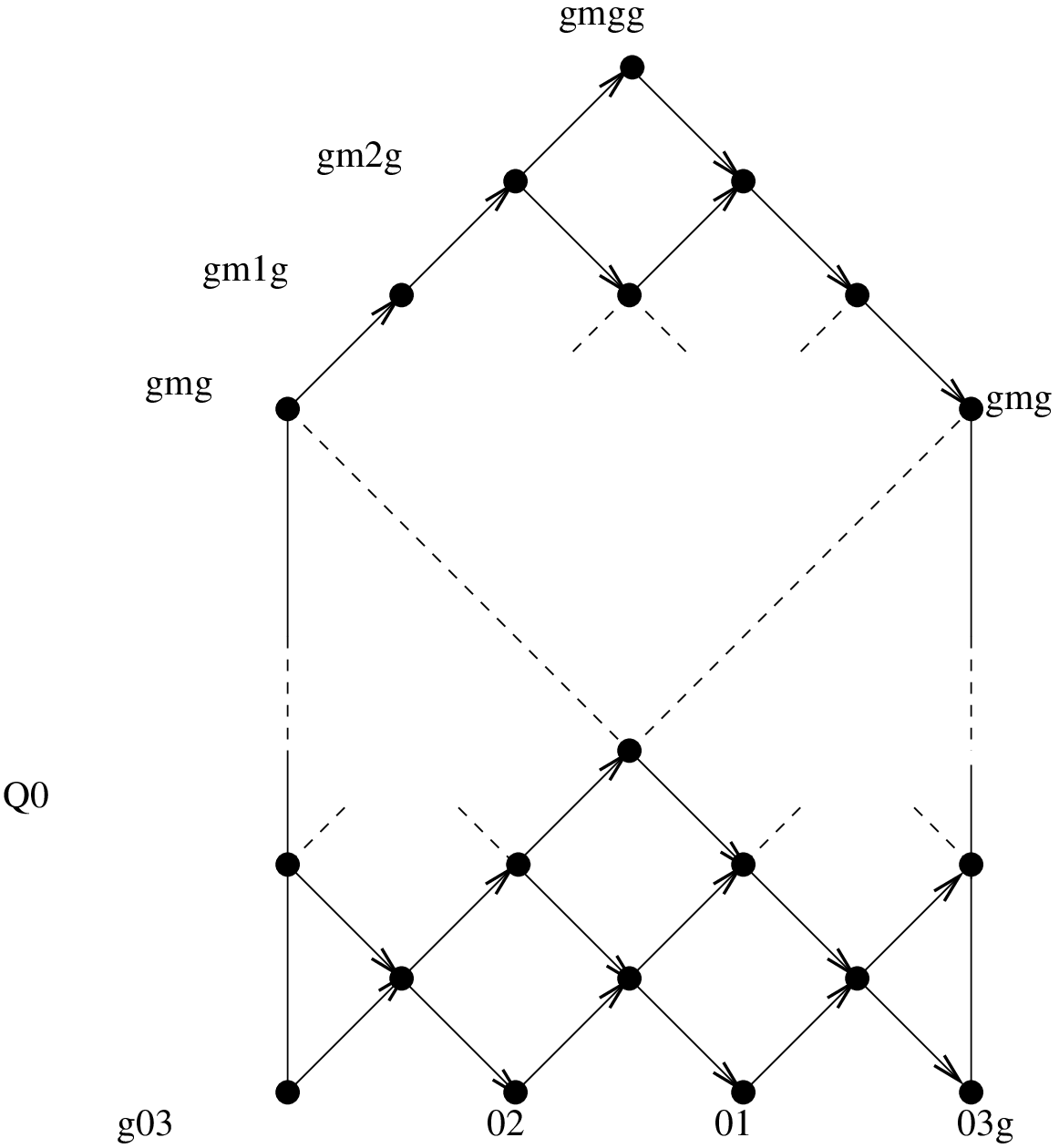}
\end{center}

As in Section~\ref{sec:postpr}, $F$ defines a functor $F_0:kQ_m^0\to R^0$; we thus define $R_m^0$ to be the category 
whose objects are $F_0X$, for all $X\in Q_m^0$. 

\begin{prop}
Let $\overline{kQ_m^0}$ be the $k$-category defined by $Q_m^0$, with the relations 
\begin{itemize}
\item 
$\pi'\rho=\rho'\pi$ for any four arrows $\pi$, $\pi'$, 
$\pi$, $\pi'$ of $Q_m^0$ in a diamond 
  $\xymatrix@R=+0.6pc @C=+0.1pc {
     & Y_1\ar[rd]^{\pi'} & \\
    X\ar[ru]^{\rho}\ar[rd]_{\pi} & & Z \\
     & Y_2\ar[ru]_{\rho'} &  
}$
\item 
$\pi_0(0,s)\rho_0(0,s')=0$ for all arrows 
  $\xymatrix@R=+0.5pc @C=+0.6pc {
     & (1,s') \ar[rd]^{\pi_0(0,s)} & \\
    (0,s')\ar[ru]^{\rho_0(0,s')}  & & (0,s)
}$
\end{itemize}
Then $F_0$ induces an isomorphism 
$\overline{F_0}:\overline{kQ_m^0}\to R_m^0$. 
\end{prop}

For $R^{\infty}$, the construction is similar to the one of $R^0$, we leave out the steps and 
give the result immediately: 

Let $Q^{\infty}$ be the quiver with the vertices $(r,s)_{\infty}$, $r\in\mathbb{N}$, $s=0,\dots,h-1$, and 
arrows $\pi_{\infty}(r,s):(r+1,s-1)_{\infty}\to (r,s)_{\infty}$ for $s=1,\dots, h-1$ and 
$\pi_{\infty}(r,s):(r+1,h-1)_{\infty} \to (r,s)_{\infty}$ for $s=0$, as well as 
$\rho_{\infty}(r,s):(r,s)_{\infty}\to (r+1,s)_{\infty}$, for $r\in\mathbb{N}$. 
Let $Q_m^{\infty}$ be the full subquiver of $\mathbb{Q}^{\infty}$ 
on the vertices $(r,s)_{\infty}$, for $s=0,\dots, h-1$ and $r=0,\dots, h(m+1)-s$. 

In the running example, $Q_2^{\infty}$ looks as follows: 
\begin{center}
\psfragscanon
\psfrag{Q2-inf}{\tiny $Q_2^{\infty}$:}
\psfrag{00}{\tiny $_{(0,1)}$}
\psfrag{01}{\tiny $_{(0,2)}$}
\psfrag{10}{\tiny $_{(1,0)}$}
\psfrag{11}{\tiny $_{(1,1)}$}
\psfrag{21}{\tiny $_{(2,1)}$}
\psfrag{20}{\tiny $_{(2,0)}$}
\psfrag{40}{\tiny $_{(4,0)}$}
\psfrag{50}{\tiny $_{(5,0)}$}
\psfrag{60}{\tiny $_{(6,0)}$}
\psfrag{51}{\tiny $_{(5,1)}$}
\psfrag{r00}{\tiny $_{\rho(0,0)}$}
\psfrag{r01}{\tiny $_{\rho(0,1)}$}
\psfrag{p01}{\tiny $_{\pi(0,1)}$}
\psfrag{p00}{\tiny $_{\pi(0,0)}$}
\includegraphics[scale=.6]{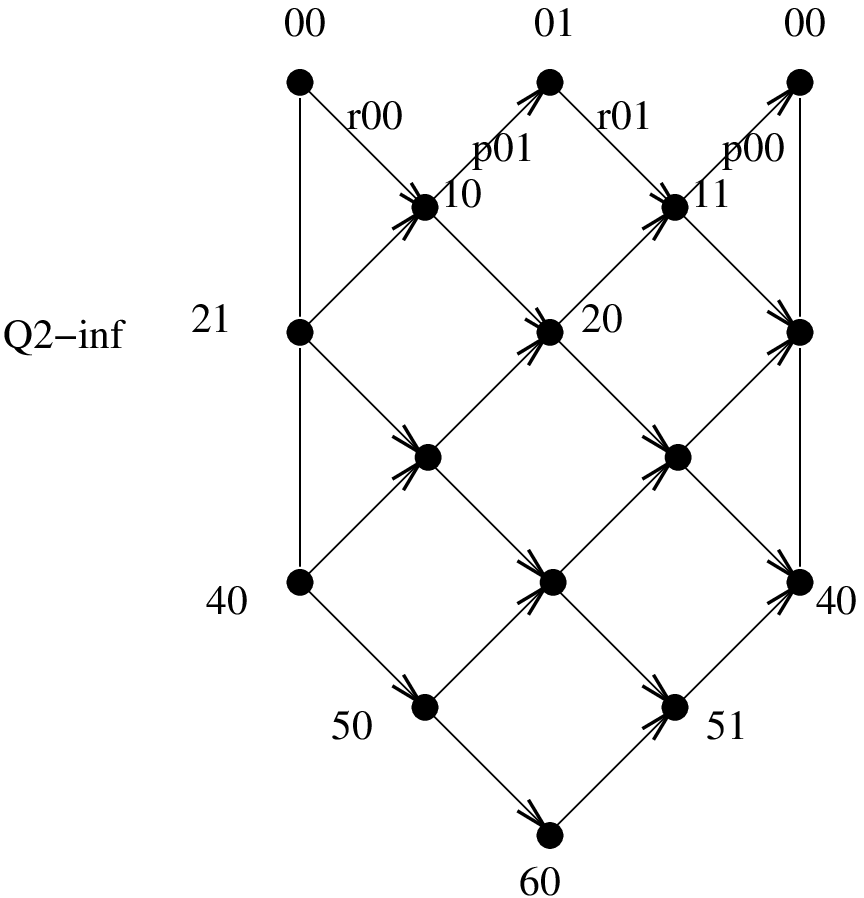}
\end{center}

We define a functor $F_{\infty}:kQ_m^{\infty}\to R^{\infty}$ as follows: Let $(r,s)_{\infty}\in Q_m^{\infty}$ and 
$$
\begin{array}{ll}
p_s = \left\{\begin{array}{ll} 0 & \mbox{if $s=0$} \\ g+s &  \mbox{if $s=1,\dots, h-1$}\end{array} \right.  & q_{r,s} = g + {\rm MOD}(r+s,h) + (n+1)\cdot{\rm DIV}(r+s,h)
\end{array}
$$
With this, we set 
$$
F_{\infty}(r,s)_{\infty}=V_{(p_s,q_{r,s})}\in R^{\infty}
$$
For any arrow $\rho_{\infty}(r,s):(r,s)_{\infty}\to (r+1,s)_{\infty}$ of $Q_m^{\infty}$ we define a morphism 
$F_{\infty}\rho_{\infty}(r,s):V_{(p_s,q_{r,s})}\to V_{(p_s,q_{r+1,s})}$ as follows: 
$$
F_{\infty}\rho_{\infty}(r,s)(e_t)=e_t\quad\mbox{for $t=p_s,\dots,q_{r,s}$} 
$$
For any arrow $\pi_{\infty}(r,s):(r+1,s-1)\to(r,s)$ of $Q_m^{\infty}$ we define a morphism 
$F_{\infty}\pi_{\infty}(r,s):V_{(p_{s-1},q_{r+1,s-1})}\to V_{(p_s,q_{r,s})}$ as follows: \\
If $s\in \{1,\dots, h-1\}$: 
$$
F_{\infty}\pi_{\infty}(r,s)(e_t)= \left\{  
                         \begin{array}{ll} 
                              0 & \mbox{ for } q_{s-1}\le t\le p_s  \\
                              e_t & \mbox{ for } t=p_q,\dots q_{r+s,s-1}                                
                          \end{array}   \right. 
$$
And if $s=0$ we define $F_{\infty}\pi_{\infty}(r,0):V_{(n,q_{r+1,h-1})}\to V_{(0,q_{r,0})}$ by setting 
$$
F_{\infty}\pi_{\infty}(r,0)(e_t)= \left\{  
                         \begin{array}{ll} 
                              0 & \mbox{ if } t=n\\
                              e_{t-(n+1)} & \mbox{ for } t=n+1,\dots,q_{r+1,h-1}. 
                          \end{array}   \right. 
$$

\begin{ex}
$$
F_{\infty}(1,1)_{\infty}\stackrel{\tiny F_{\infty}\rho(1,1)}{\longrightarrow} F_{\infty}(2,1)_{\infty} \stackrel{\tiny F_{\infty}\pi(1,0)}{\longrightarrow} F_{\infty}(1,0)_{\infty}: 
\psfragscanon
\psfrag{e0}{$_{e_0}$}
\psfrag{e1}{$_{e_1}$}
\psfrag{e2}{$_{e_2}$}
\psfrag{e3}{$_{e_3}$}
\psfrag{e4}{$_{e_4}$}
\psfrag{e5}{$_{e_5}$}
\psfrag{e6}{$_{e_6}$}
\psfrag{e7}{$_{e_7}$}
\psfrag{e8}{$_{e_8}$}
\psfrag{e9}{$_{e_9}$}
\includegraphics[scale=.6]{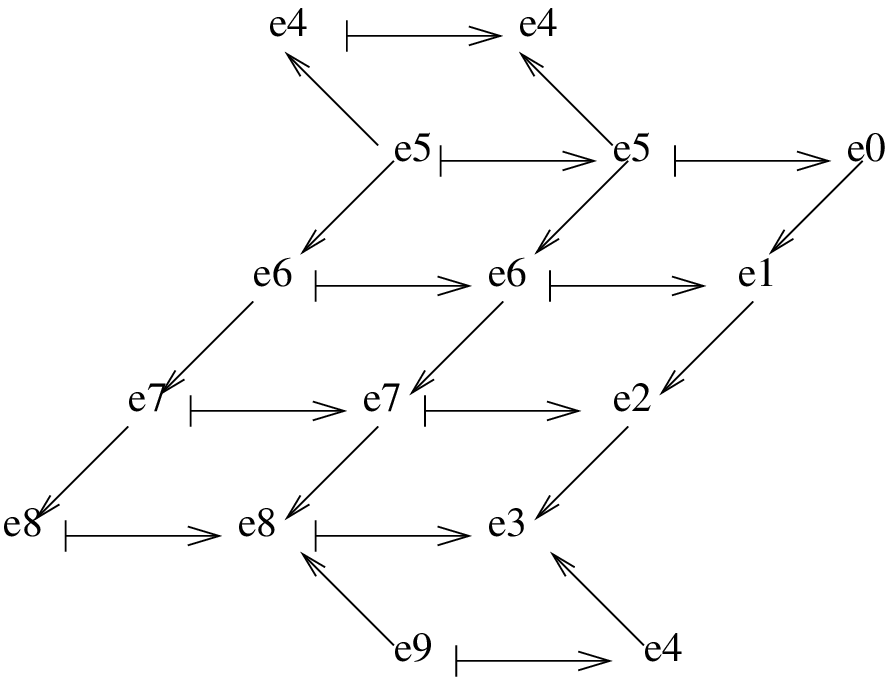}
$$
\end{ex}

\begin{prop}
Let $R_m^{\infty}$ be the category with the objects $F_{\infty}X$, for any $X\in Q_m^{\infty}$ and let $\overline{kQ_m^{\infty}}$ 
be the $k$-category defined by the quiver $Q_m^{\infty}$, subject to the relations $\pi'\rho=\rho'\pi$, for all arrows $\pi,\pi'$, $\rho, \rho'$ of 
$Q_m^{\infty}$ in a diamond 
$\xymatrix@R=+0.6pc @C=+0.1pc {
     & Y_1\ar[rd]^{\pi'} & \\
    X\ar[ru]^{\rho}\ar[rd]_{\pi} & & Z \\
     & Y_2\ar[ru]_{\rho'} &  
}$
as well as subject to the relations $\pi_{\infty}(0,s)\rho_{\infty}(0,s')=0$ for all arrows 
$\xymatrix@R=+0.5pc @C=+0.6pc {
    (0,s')\ar[rd]_{\rho_{\infty}(0,s')}  & & (0,s)\\ 
     & (1,s') \ar[ru]_{\pi_{\infty}(0,s)} & 
}$
of $Q_m^{\infty}$. Then $F_{\infty}$ induces an isomorphism 
$\overline{F_{\infty}}:\overline{kQ_m^{\infty}}\to R_m^{\infty}$. 
\end{prop}

Finally, we consider the $k$-category $R_m^{\lambda}$ for every $\lambda\in k\setminus\{0\}$ whose objects are the 
$V_d^{\lambda}$, $1\le d\le m+2$ (with $m\in \NN$). 

For this, we let $Q_m^{\lambda}$ be the quiver with vertices $(\lambda,1),\dots,(\lambda,m+2)$ and with arrows 
$\pi_{\lambda}(r):(\lambda,r+1)\to(\lambda,r)$, $\rho_{\lambda}:(\lambda,r)\to(\lambda,r+1)$ for $r=1,\dots,m+1$: 


$$
\psfragscanon
\psfrag{l1}{\tiny $_{(\lambda,1)}$}
\psfrag{l2}{\tiny $_{(\lambda,2)}$}
\psfrag{lm1}{\tiny $_{(\lambda,m+1)}$}
\psfrag{lm2}{\tiny $_{(\lambda,m+2)}$}
\psfrag{Qml}{$Q_m^{\lambda}$:}
\psfrag{r1}{\tiny $\rho_{\lambda}(1)$} 
\psfrag{rm1}{\tiny $\rho_{\lambda}(m+1)$} 
\psfrag{p1}{\tiny $\pi_{\lambda}(1)$} 
\psfrag{pm1}{\tiny $\pi_{\lambda}(m+1)$} 
\includegraphics[scale=.6]{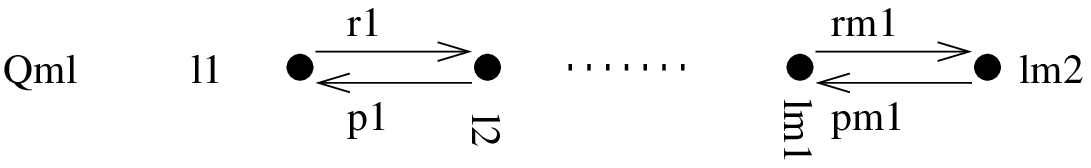}
$$

\vskip 5pt

We define a functor $F_{\lambda}:kQ_,^{\lambda}\to R_m^{\lambda}$ by setting 
$$
F_{\lambda}(\lambda,r)=V_r^{\lambda}\quad r=1,\dots,m+2
$$
with morphisms (for $r=1,\dots, m+1$): 
$$
\begin{array}{ll}
F_{\lambda}\pi_{\lambda}(r) & \mbox{given by }\begin{pmatrix} {\bf 1}_r & 0\end{pmatrix}\in k^{r\times r+1} \mbox{ in every vertex of $K$}\\
F_{\lambda}\rho_{\lambda}(r) & \mbox{given by }\begin{pmatrix} 0 \\ {\bf 1}_r \end{pmatrix}\in k^{r+1\times r} \mbox{in every vertex of $K$} 
\end{array}
$$

\begin{prop}
$F_{\lambda}$ induces an isomorphism between the $k$-category defined by $Q_m^{\lambda}$, subject to the relations 
\begin{enumerate}
\item
$\pi_{\lambda}(r)\rho_{\lambda}(r)= \rho_{\lambda}(r+1)\pi_{\lambda}(r+1)$, $r=1,\dots,m$ 
\item
$\pi_{\lambda}(1)\rho_{\lambda}(1)=0$. 
\end{enumerate}
and the $k$-category $R_m^{\lambda}$. 
\end{prop}

\begin{rem}
Condition (1) in the proposition above can also be written in the form 
$$
\pi'\rho=\rho'\pi,\ \mbox{ for all arrows of $Q_m^{\lambda}$ of the form }   
   \xymatrix@R=+0.6pc @C=+0.1pc {
     & Y_1\ar[rd]^{\pi'} & \\
    X\ar[ru]^{\rho}\ar[rd]_{\pi} & & Z \\
     & Y_2\ar[ru]_{\rho'} &  
}
$$
since $X=Z$ is allowed. 
\end{rem}

%
\section{Main Theorem}
%

For $m\in \NN$ let $Q_m$ be the quiver whose vertices are the vertices of the quivers 
$Q_m^P$, $Q_m^I$ and $Q_{2m(n+1)}^{\sigma}$, for $\sigma\in k\cup\{\infty\}$, the arrows 
are all the arrows of these three quivers, with in additional the ``connecting'' arrows 
$\iota_0(x)$, $\kappa_0(x)$ for $x=0,\dots,g$, $\iota_{\infty}(y)$, $\kappa_{\infty}(y)$ for 
$y=0,g,g+1,\dots, n$ and $\iota_{\lambda}$, $\kappa_{\lambda}$ for $\lambda\in k\setminus\{0\}$. 
Their definition is apparent from the figure below of $Q_m$ below. (Remark by KB: 
The example $g=h=1$ appears in Examples 5 and 6 in Chapter 8 of~\cite{gr}) 

\begin{figure}[ht]
\begin{center}
\psfragscanon
\psfrag{QI}{$Q_m^I$}
\psfrag{QP}{$Q_m^P$}
\psfrag{Q0}{$Q_{2m(n+1)}^{0}$}
\psfrag{QL}{$Q_{2m(n+1)}^{\lambda}$}
\psfrag{Q-inf}{$Q_{2m(n+1)}^{\infty}$}
\psfrag{i00}{\tiny $_{\iota_0(0)}$}
\psfrag{i01}{\tiny $_{\iota_0(1)}$}
\psfrag{i02}{\tiny $_{\iota_0(2)}$}
\psfrag{i03}{\tiny $_{\iota_0(3)}$}
\psfrag{k00}{\tiny $_{\kappa_0(0)}$}
\psfrag{k01}{\tiny $_{\kappa_0(1)}$}
\psfrag{k02}{\tiny $_{\kappa_0(2)}$}
\psfrag{k03}{\tiny $_{\kappa_0(3)}$}
\psfrag{ki0}{\tiny $_{\kappa_{\infty}(0)}$}
\psfrag{ki3}{\tiny $_{\kappa_{\infty}(3)}$}
\psfrag{ki4}{\tiny $_{\kappa_{\infty}(4)}$}
\psfrag{ii0}{\tiny $_{\iota_{\infty}(0)}$}
\psfrag{ii3}{\tiny $_{\iota_{\infty}(3)}$}
\psfrag{ii4}{\tiny $_{\iota_{\infty}(4)}$}
\psfrag{00}{\tiny $_{(0,0)}$}
\psfrag{01}{\tiny $_{(0,1)}$}
\psfrag{02}{\tiny $_{(0,2)}$}
\psfrag{03}{\tiny $_{(0,3)}$}
\psfrag{04}{\tiny $_{(0,4)}$}
\psfrag{ri}{\tiny $_{\rho_{\infty}}$}
\psfrag{pi}{\tiny $_{\pi_{\infty}}$}
\psfrag{r0}{\tiny $_{\rho_{0}}$}
\psfrag{p0}{\tiny $_{\pi_{0}}$}
\psfrag{il}{\tiny $_{\iota_{\lambda}}$}
\psfrag{kl}{\tiny $_{\kappa_{\lambda}}$}
\psfrag{pl1}{\tiny $_{\pi_{\lambda}(1)}$}
\psfrag{rl1}{\tiny $_{\rho_{\lambda}(1)}$}
\psfrag{20m0}{\tiny $_{(20m,0)}$}
\psfrag{6m3}{\tiny $_{(6m,3)}$}
\psfrag{6m0}{\tiny $_{(6m,0)}$}
\psfrag{30m30}{\tiny $_{(30m,3)_0}$}
\psfrag{30m3I}{\tiny $_{(30m,3)_I}$}
\psfrag{6m3I}{\tiny $_{(6m,3)_I}$}
\includegraphics[scale=.5]{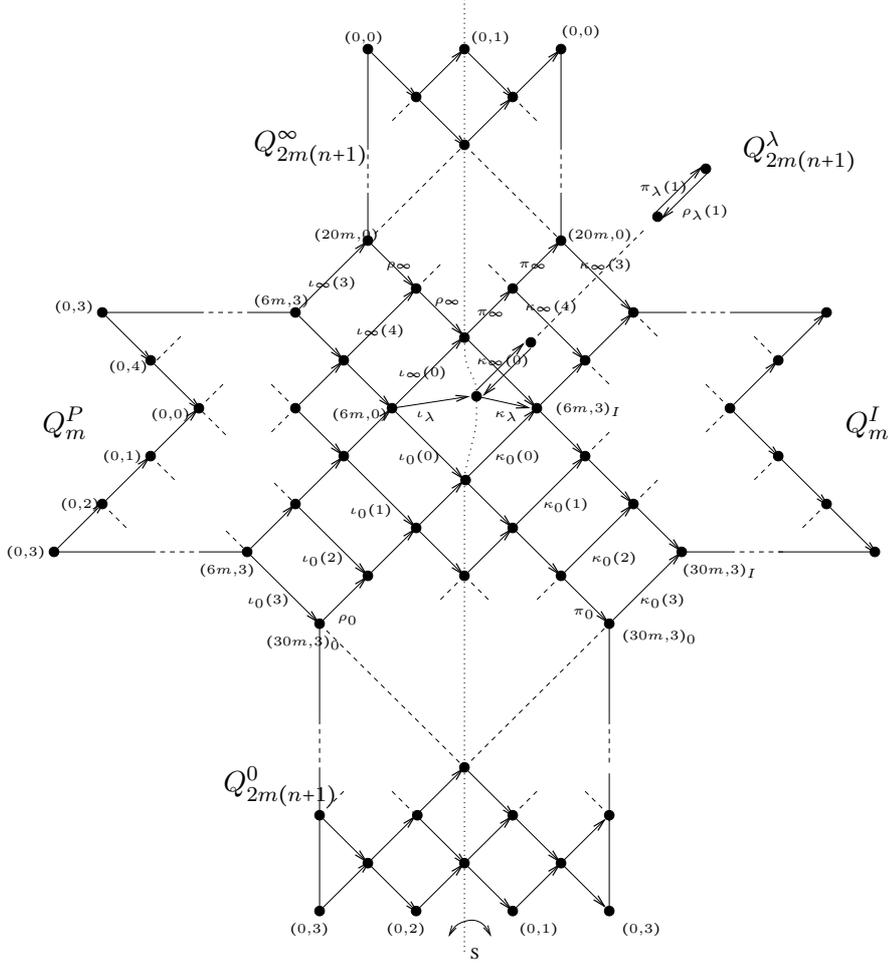}
\caption{$Q_m$ for running example: $g=3$, $n=4$ ($ghm=6m$, $g2m(n+1)=30m$, $h2m(n+1)=20m$)}
\label{fig:quiver-Qm}
\end{center}
\end{figure}

By $\mathcal{J}_mK$ we denote the full subcategory of $\rep K$ whose objects are the union of the objects of 
$P_m$, $I_m$, $R_{2m(n+1}^{\sigma}$, $\sigma\in k\cup\{\infty\}$ (for $m\in \NN$).

\begin{rem}
To any $V\in\rep K$ there exists $m\in\NN$ such that $V$ is isomorphic to a direct sum of representations 
from $\mathcal{J}_mK$. 
\end{rem}

Our goal is to find an isomorphism $\Phi_m$ (for any $m\in\NN$) between the $k$-category defined by $Q_m$ and certain relations 
and the $k$-category $\mathcal{J}_mK$. 
For this, we first define a functor $\Phi_m:kQ_m\to \mathcal{J}_mK$: 

Let $X$ be a vertex of $Q_m$. We set 
$$
\Phi_mX := \left\{ \begin{array}{ll} F_PX & \mbox{if $X\in Q_m^P$} \\
                                                F_IX & \mbox{if $X\in Q_m^I$} \\
                                                F_{\sigma} X & \mbox{if $X\in Q_{2m(n+1)}^{\sigma}$},\quad \sigma\in k\cup\{\infty\}
                   \end{array} \right.
$$

Analoguosly, if $\gamma$ is an arrow of $Q_m$, we set 
$$
\Phi_m\gamma := \left\{ \begin{array}{ll} F_P\gamma & \mbox{if $\gamma$ is an arrow of  $Q_m^P$} \\
                                                F_I\gamma & \mbox{if $\gamma$ is an arrow of $Q_m^I$} \\
                                                F_{\sigma} \gamma & \mbox{if $\gamma$ is in $Q_{2m(n+1)}^{\sigma}$},\quad \sigma\in k\cup\{\infty\}
                   \end{array} \right.
$$
We also need to define $\Phi_m$ on the connecting arrows. 
\begin{itemize}
\item
For any arrow $\iota_0(x):X\to Y$ in $Q_m$, $x=0,\dots,g$, define the morphism $\Phi_m\iota_0(x):V_{(p,q)}\to V_{(p',q')}$ (where 
$V_{(p,q)}=\Phi_mX$ and $V_{(p',q')}=\Phi_m Y$) by setting 
$$
\Phi_m\iota_0(x)(e_t)=e_t,\quad t=p,\dots,q
$$
\item
For any arrow $\iota_{\infty}(x):X\to Y$ in $Q_m$, $x=0,g, \dots,n$ we define a morphism $\Phi_m\iota_{\infty}(x):\Phi_mX=V_{(p,q)}\to V_{p',q')}=\Phi_mY$ 
by 
$$
\Phi_m\iota_{\infty}(x)(e_{q-t})=e_{q'-t}\ \mbox{for } t=0,\dots, q-p.
$$ 
\item 
For $\lambda\in k\setminus\{0\}$ we define a morphism $\Phi_m\iota_{\lambda}:\Phi_m(ghm,0)_p=V_w\to \Phi_m(\lambda,2m(n+1)+2)$ by 
giving the matrices w.r.t. the basis $e_g,\dots,e_{(n+1)(m(n+1)+1)+g}$ of $E_w$ and the canonical basis of the spaces $k^{2m(n+1)+2}$. 
For this, 
we let  (for $d\in \NN$) $D_d{\lambda}\in k^{2m(n+1)+2\times d}$ be the matrix 
$$
D_d(\lambda)_{ij}:=\left\{\begin{array}{ll} 0 & \mbox{for $i>j$}, \\ {{j-1}\choose{i-1} }\lambda^{j-i}& \mbox{else}\end{array}\right.  \quad
D_d(\lambda): \begin{pmatrix} 1 & \lambda & \lambda^2 & \cdots & \lambda^{d-1} \\ & 1 & 2\lambda & & \vdots \\  & & 1 & & \vdots \\  & & & \ddots  \\ 
& & & & 1 \\  \phantom{} & \\ & \phantom{} \end{pmatrix}
$$
Taking into account the dimensions of the vector spaces $\Phi_m(ghm,0)_P(x)$ we define $\Phi_m\iota_{\lambda}(x)$ as 
$$
\left\{
\begin{array}{ll} D_{m(n+1)+2}(\lambda) & \mbox{if $x=g$} \\ D_{m(n+1)+1}(\lambda) & x=0,\dots,g-1,g+1,\dots,n \end{array}\right. 
$$
Then $\Phi_m\iota_{\lambda}$ is a morphism. 
\end{itemize}
To define $\Phi_m$ on the remaining arrows of $Q_m$, we use the construction from Section~\ref{sec:pre-inj}. Let $S$ be the permutation of the 
vertices of $Q_m$ that corresponds to a reflection along the vertical (dashed) line in Figure~\ref{fig:quiver-Qm}. 
For every vertex $X$ of $Q_m^P$ or of $Q_m^I$ we use the definition of Section~\ref{sec:pre-inj}; 
for $(r,s)_0\in Q_{2m(n+1)}^0$ let $s_0\in\{1,\dots,g\}$ be such that $s_0\equiv r-s$ $\mod g$. Then we set $S(r,s)_0=(r,s_0)_0$. 
For $(r,s)_{\infty}\in Q_{2m(n+1)}^{\infty}$ let $s_{\infty}\in \{0,\dots, h-1\}$ be such that $s_{\infty}=-r-s$ $\mod h$; 
then set $S(r,1)_{\infty}=(r,s_{\infty})_{\infty}$; 
finally, we define $S(\lambda,r)=(\lambda,r)$ for all $(\lambda,r)\in Q_{2m(n+1)}^{\lambda}$, $\lambda\in k\setminus\{0\}$. 

For every arrow $\gamma:X\to Y$ of $Q_m$ let $S\gamma$ be the arrow $S\gamma:SY\to SY$ of $Q_m$; then we have 
$$
\begin{array}{l}
S\iota_0(x)=\kappa_0(x)\quad\mbox{for $x=0,\dots,g$}, \\
S\iota_{\infty}(y)=\kappa_{\infty}(y) \quad\mbox{for $y=0,g,\dots,n$}\\
S\iota_{\lambda}=\kappa_{\lambda}\quad\mbox{for $\lambda\in k\setminus\{0\}$}. 
\end{array}
$$

\begin{prop}
For every vertex $X$ of $Q_m$ there is an isomorphism 
$$
\varphi_X:\overline{\Phi_mX}\stackrel{\sim}{\longrightarrow}\Phi_mSX
$$
\end{prop}

\begin{proof}
First let $X$ be a vertex of one of the quivers $Q_m^P$, $Q_{2m(n+1)}^0$, $Q_{2m(n+1)}^{\infty}$ 
and let $\Phi_mX=V_{(p,q)}$. 

By Section~\ref{sec:pre-inj} we know that there is an isomorphism $\varphi=\varphi_X:\overline{V_{(p,q)}}\stackrel{\sim}{\longrightarrow} V_{(p',q')}$ 
given by $\varphi_X(\check{e}_{p+t})=e_{q'-t}$, for $t=0,\dots,q-p$, for integers $p',q'$ with $p'\equiv Gq$ 
$\mod n+1$, $q'\equiv Gp$ $\mod n+1$ and 
$q'-p'=q-p$. \\
We have $\Phi_m SX=V_{(p',q')}$. 

\vskip 5pt 

For $X\in Q_m^I$, the claim follows from the fact that for any $V\in\rep K$ there is an isomorphism between $\overline{\overline{V}}$ and $V$. 

\vskip 5pt 

Now let $X=(\lambda,r)\in Q_{2m(n+1)}^{\lambda}$, $\lambda\in k\setminus\{0\}$. We give an isomorphism 
$\varphi_X^{-1}:V_r^{\lambda}\stackrel{\sim}{\longrightarrow} \overline{V_r^{\lambda}}$, we then let $\varphi_X$ be the inverse morphism of this. 

We describe $\varphi_X^{-1}$ using the matrices representing it with respect to the canonical bases $e_1,\dots,e_r$ of $k^r$ and $\check{e}_1,\dots,\check{e}_r$ 
of $(k^r)^{\perp}$: 
Using $\theta_r:=\begin{pmatrix} 0 & & 1 \\  & \udots \\ 1 & & 0 \end{pmatrix}\in k^{r\times r}$ let $\varphi_X^{-1}$ be given by 
$$
\left\{
\begin{array}{ll}
\theta_r & \mbox{if } x=0,g,g+1,\dots, n \\ 
\lambda\theta_r + \theta_r J_r & \mbox{ if } x=1,\dots,g-1 \quad \quad \mbox{ where }J_r=\begin{pmatrix} 0 & & & 0 \\ 1 & \ddots & \\  & \ddots & \ddots \\ 0 & &1 & 0  \end{pmatrix}
\in k^{r\times r}
\end{array}
\right.
$$
This proves the claim. 

For any connecting arrow $\gamma:X\to Y$ of $Q_m$ with $X\in Q_m^P$ and $Y\in Q_{2m(n+1)}^{\sigma}$, $\sigma\in k\cup\{\infty\}$, we define 
$$
\Phi_mS\gamma:\Phi_mSY\to \Phi_mSX \quad \mbox{as follows }\quad  \Phi_mS\gamma:=\varphi_X\circ\overline{\Phi_m\gamma}\circ\varphi_Y^{-1} 
$$
Hence we have associated a morphism to any arrow of $Q_m$ and thus defined the functor $\Phi_m:kQ_m\to  \mathcal{J}_mK$ uniquely. 
\end{proof}

To keep the description of the relations in the following statement simple, we assume $n\ge 2$ and introduce some notation to allow us 
to abbreviate the relations: 

\begin{center}
\psfragscanon
\psfrag{QI}{$Q_m^I$}
\psfrag{QP}{$Q_m^P$}
\psfrag{Q0}{$Q_{2m(n+1)}^{0}$}
\psfrag{Q-inf}{$Q_{2m(n+1)}^{\infty}$}
\psfrag{ghm0P}{\tiny $_{(ghm,0)_P}$}
\psfrag{ghm0I}{\tiny $_{(ghm,0)_I}$}
\psfrag{i00}{\tiny $_{\iota_0(0)}$}
\psfrag{k00}{\tiny $_{\kappa_0(0)}$}
\psfrag{ki0}{\tiny $_{\kappa_{\infty}(0)}$}
\psfrag{ii0}{\tiny $_{\iota_{\infty}(0)}$}
\psfrag{ghma}{\tiny $(ghm,\alpha_{g-1})$}
\psfrag{ghmgP}{\tiny $(ghm,g)_P$}
\psfrag{h2m0}{\tiny $(2hm(n+1),0)$}
\psfrag{ghmgI}{\tiny $(ghm,g)_I$}
\psfrag{g2mg0}{\tiny $(g2m(n+1),g)_0$}
\psfrag{agP}{\tiny $\alpha_P^g$}
\psfrag{bhP}{\tiny $\beta_P^h$}
\psfrag{rh}{\tiny $\rho_{\infty}^h$}
\psfrag{ph}{\tiny $\pi_{\infty}^h$}
\psfrag{bhI}{\tiny $\beta_I^h$}
\psfrag{agI}{\tiny $\alpha_I^g$}
\psfrag{pg}{\tiny $\pi_0^g$}
\psfrag{rg}{\tiny $\rho_0^g$}
\includegraphics[scale=.6]{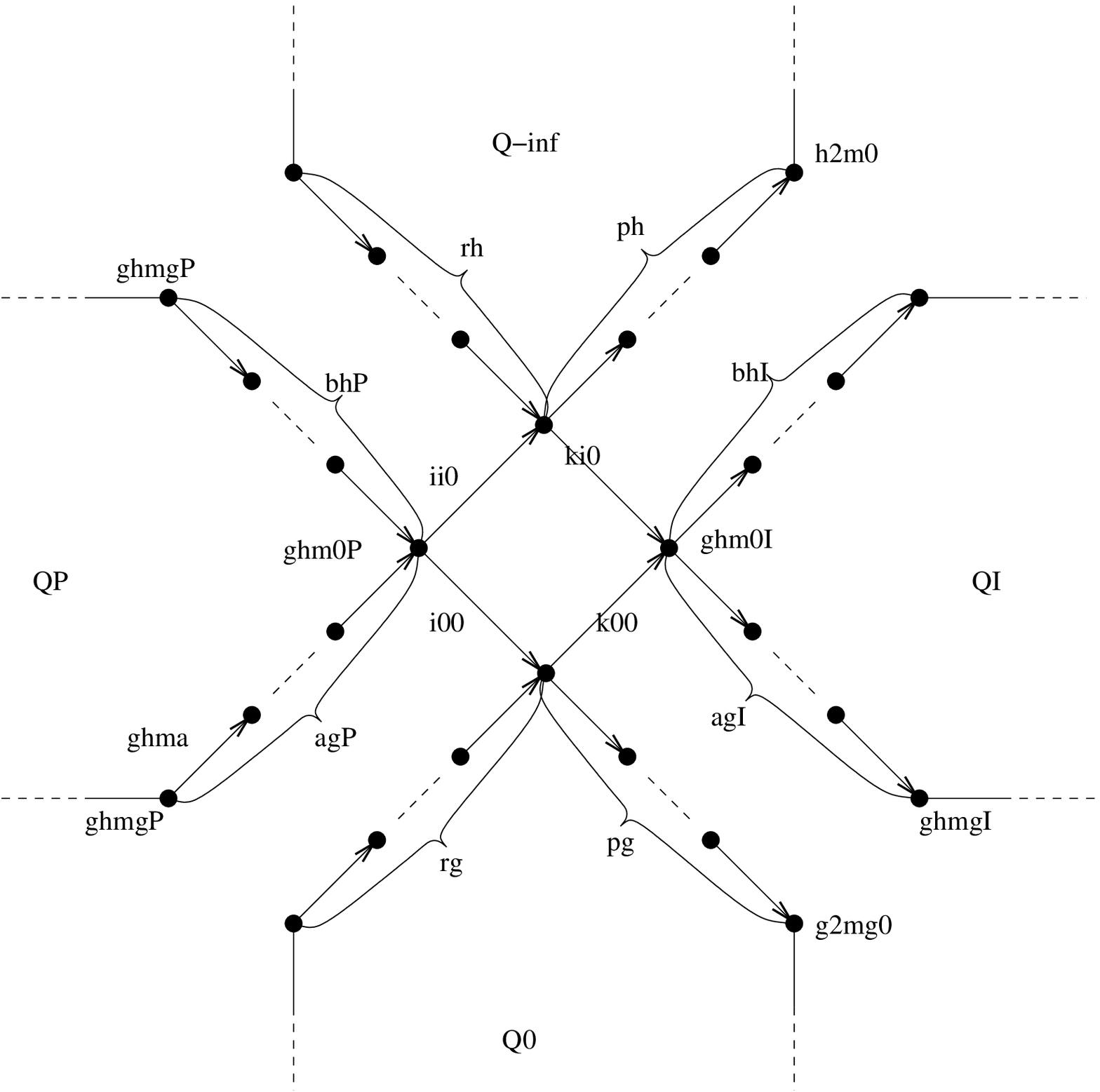}
\end{center}

\begin{itemize}
\item 
$\beta_P^g$ denotes the path of length $g$ from $(ghm,g)_P$ to $(ghm,0)_P$ composed by the arrows $(ghm,\beta_X)_P$. 
\item 
$\alpha_P^h$ denotes the path of length $h$ from $(ghm,g)_P$ to $(ghm,0)_P$ composed by the arrows $(ghm,\alpha_X)_P$. 
\end{itemize}
The definitions of $\alpha_I^g$, $\beta_I^g$, $\pi_0^g$, $\rho_0^g$, $\pi_{\infty}^h$ and $\rho_{\infty}^h$ can be understood from 
the picture above. 

The ``lowest'' vertices of the tubes $Q_{2m(n+1)}^{\sigma}$, $\sigma\in k\cup\{\infty\}$, i.e. the vertices \\
$(0,s)_0\in Q_{2m(n+1)}^0$, $s=1,\dots,g$ \\
$(0,s')_{\infty}$, $s'=0,\dots, h-1$\footnote{The $s'$ were equal to $1,\dots,h$ on page 12.}\\
$(\lambda,1)\in Q_{2m(n+1)}^{\lambda}$ 
are called the vertices at the mouth of the tubes. 
Finally, we write \\
$\varepsilon_{\lambda}=\rho_{\lambda}(2m(n+1)+1)\pi_{\lambda}(2m(n+1)+1)$ for $\lambda\in k\setminus\{0\}$.

\begin{thm}[Hauptsatz]\label{thm:hauptsatz}
Let $\overline{kQ_m}$ be the $k$-category defined by $Q_m$ and the following relations 
\begin{itemize}
\item[(a)] 
$\gamma'\gamma=\delta'\delta$ for all arrows $\gamma$, $\gamma'$, $\delta$, $\delta'$ of $Q_m$  
in a diamond from $X$ to $Z$ with $X\ne (ghm,0)_P$. 

\vskip 5pt

\item[(b)] 
$\pi_{\sigma}\rho_{\sigma}=0 $ for all $\sigma\in k\cup \{\infty\}$ and all arrows  
$\pi_{\sigma}$, $\rho_{\sigma}$ of $Q_{2m(n+1)}^{\sigma}$ 
of the form 
$X^*\stackrel{\rho_{\sigma}}{\longrightarrow} Y\stackrel{\pi_{\sigma}}{\longrightarrow} Z^*$ 
with base points $X^*$ and $Z^*$. 

\vskip 5pt

\item[(c1)]
$\iota_0(g) = \pi_0^g\iota_0(0)\beta_P^h$ \hskip 60pt 
$\iota_{\infty}(g) = \pi_{\infty}^h\iota_{\infty}(0)\alpha_P^g$. 

\vskip 5pt

\item[(c2)] 
$\kappa_0(g) = \rho_0^g\kappa_0(0)\beta_I^h$ \hskip 57pt 
$\kappa_{\infty}(g) = \alpha_I^g \kappa_{\infty}(0)\rho_{\infty}^h$. 

\vskip 5pt

\item[(d)]
$\iota_{\lambda}\alpha_P^g = \lambda\iota_{\lambda}\beta_P^h 
 + \varepsilon_{\lambda}\iota_{\lambda}\beta_P^h$ \hskip 37pt
$\alpha_I^g\kappa_{\lambda} = \lambda\beta_I^h\kappa_{\lambda} 
 + \beta_I^h\kappa_{\lambda}\varepsilon_{\lambda}$. 

\vskip 5pt

\item[(e)]
$\kappa_0(0)\iota_0(0)(ghm,\alpha_n)_P=0$ \hskip 22pt $(ghm,\alpha_n)_I\kappa_0(0)\iota_0(0)=0$.  

\vskip 5pt

\item[(f)]
$\kappa_{\infty}(0)(\rho_{\infty}^h\pi_{\infty}^h)^j\iota_{\infty}(0) = 
\kappa_0(0)(\rho_0^g\pi_0^g)^{2m(n+1)+1-j}\iota_0(0)$. 

\vskip 5pt

\item[(g)] 
$\kappa_{\lambda}\varepsilon_{\lambda}^j\iota_{\lambda} = 
\sum_{i=0}^{j}{2m(n+1)+1-j+i \choose 2m(n+1)+1-j}\lambda^i
\kappa_0(0)(\rho_0^g\pi_0^g)^{j-i}\iota_0(0)$
\end{itemize}

\vskip 5pt

\noindent
with $\lambda\in k\setminus \{0\}$ and $j=0,\,\dots, \, 2m(n+1)+1$. \\
Then $\Phi_m$ induces an isomorphism 
$$
\overline{\Phi_m}:\overline{kQ_m}\to \mathcal{J}_mK
$$
\end{thm}

\section*{Remarks} 
The diploma thesis ends here. This section contains a few remarks concerning the relations 
appearing in Theorem~\ref{thm:hauptsatz}. 

\begin{rem}
In (a), we have all the mesh relations with two middle vertices. This includes meshes between 
different components. The latter are indicated by dashed red lines Figure~\ref{fig:quiver-again}. 

(b) shows the relations at the mouths of the tubes. 
So these are the mesh relations involving only one middle vertex. 

In (c1) and (c2), we have relations for arrows between $Q_m^P$ and $Q_{2m(n+1)}^{\infty}$ 
or $Q_{2m(n+1)}^{\infty}$ (respectively), as well as between any of these two tubes and $Q_m^I$. 
The paths/arrows start at the vertex $(ghm,g)_P$ and end at $(g2m(n+1),g)_0$ resp. at 
$(g2m(n+1),g)_{\infty}$ (the two relations at the left) or they start at $(g2m(n+1),g)_0$ (resp. 
at $(g2m(n+1),g)_{\infty}$) and end at $(ghm,g)_I$ (the two relations at the right). 

The relations in (d) link the homogeneous tubes with the post-projective and with the 
pre-projective component. \\
The relation to the left says that when you first compose 
all the $(ghm,\alpha_i)_P$ (starting at $(ghm,g)_P$ ($g$ of them) 
and then compose this with the arrow $\iota_{\lambda}$ 
to get to $Q_{2m(n+1)}^{\lambda}$, it is the same as the sum of two paths, both starting
all the $(ghm,\beta_j)_P$ first ($h$ of them) and the arrow $\iota_{\lambda}$, one is the multiple 
of this by $\lambda$, the other the uses $\varepsilon_{\lambda}$, a path that just goes 
one arrow in this homogeneous tube (from vertex $(\lambda, 2m(n+1)+2$ and back). \\
The relation to the right is dual to it: using $\kappa_{\lambda}$ to go from the homogeneous 
tube to $R_m^I$ and then the composition of all the $(ghm,\alpha_i)_I$ ($g$ of them) is the same 
as the sum of two paths, one starting with $\varepsilon_{\lambda}$, then out of the tube and 
the composition of all the $(ghm,\beta_j)_I$ ($h$ of them), the other the multiple of $\lambda$ of the 
path that just goes out of the tube and then does the composition 
of all the $(ghm,\beta_j)_I$ ($h$ of them).

(e) 
In the diploma thesis, relation (e) uses $\beta_n$ twice, they are replaced by $\alpha_n$. 
These two relations link 
paths from the post-projective component $Q_m^P$ through $Q_{2m(n+1)}^0$ to the pre-injective 
component. 
In the left equation, the last of the $\alpha$'s, s used, then the connecting 
arrows to $Q_{2m(n+1)}^0$ and from there to $Q_m^I$. This composition is zero. 
Dually, the composition of these two connecting arrows with the first of the $\alpha$'s, i.e. 
with $(ghm,\alpha_n)_I$ is zero.

\noindent
Note that the two paths 
$$
\kappa_{\infty}(0)\iota_{\infty}(0)(ghm,\beta_0)_P,\hskip 5pt (ghm,\beta_0)_I\kappa_{\infty}(0)\iota_{\infty}(0)
$$
are also zero. This follows from (e) and (f) (with $j=1$), using the diamond relations of (a) iteratedly to 
push the path all the way down to include a triangle at the mouth of the tube (and then use (b)). 

\noindent
We also note that the four ``other'' compositions 
$$
\begin{array}{ll}
\kappa_0(0)\iota_0(0)(ghm,\beta_0)_P &  (ghm,\beta_0)_I\kappa_0(0)\iota_0(0), \\ 
\kappa_{\infty}(0)\iota_{\infty}(0)(ghm,\alpha_n)_P,&  (ghm,\alpha_n)_I\kappa_{\infty}(0)\iota_{\infty}(0)
\end{array}
$$
are not zero! 

(f) 
Relates paths from $(ghm,0)_P$ to $(ghm,0)_I$, passing around the `top layer' of 
$Q_{2m(n+1)}^{\infty}$ resp. around $Q_{2m(n+1)}^0$ several times (adding up to $2m(n+1)+1$). 

(g) 
Relates paths from $(ghm,0)_P$ passing through a homogeneous tube, going into this 
homogeneous tube (with $\varepsilon$ used $j$ times) and on to $(ghm,0)_I$ with 
a sum of paths passing through $Q_m^0$, going around this tube (at the `top layer') several times, 
and multiplying by a scalar. 
\end{rem}

\newpage
\begin{figure}[ht]
\begin{center}
\psfragscanon
\psfrag{QI}{$Q_m^I$}
\psfrag{QP}{$Q_m^P$}
\psfrag{Q0}{$Q_{2m(n+1)}^{0}$}
\psfrag{QL}{$Q_{2m(n+1)}^{\lambda}$}
\psfrag{Q-inf}{$Q_{2m(n+1)}^{\infty}$}
\psfrag{i00}{\tiny $_{\iota_0(0)}$}
\psfrag{i01}{\tiny $_{\iota_0(1)}$}
\psfrag{i02}{\tiny $_{\iota_0(2)}$}
\psfrag{i03}{\tiny $_{\iota_0(3)}$}
\psfrag{k00}{\tiny $_{\kappa_0(0)}$}
\psfrag{k01}{\tiny $_{\kappa_0(1)}$}
\psfrag{k02}{\tiny $_{\kappa_0(2)}$}
\psfrag{k03}{\tiny $_{\kappa_0(3)}$}
\psfrag{ki0}{\tiny $_{\kappa_{\infty}(0)}$}
\psfrag{ki3}{\tiny $_{\kappa_{\infty}(3)}$}
\psfrag{ki4}{\tiny $_{\kappa_{\infty}(4)}$}
\psfrag{ii0}{\tiny $_{\iota_{\infty}(0)}$}
\psfrag{ii3}{\tiny $_{\iota_{\infty}(3)}$}
\psfrag{ii4}{\tiny $_{\iota_{\infty}(4)}$}
\psfrag{00}{\tiny $_{(0,0)}$}
\psfrag{01}{\tiny $_{(0,1)}$}
\psfrag{02}{\tiny $_{(0,2)}$}
\psfrag{03}{\tiny $_{(0,3)}$}
\psfrag{04}{\tiny $_{(0,4)}$}
\psfrag{ri}{\tiny $_{\rho_{\infty}}$}
\psfrag{pi}{\tiny $_{\pi_{\infty}}$}
\psfrag{r0}{\tiny $_{\rho_{0}}$}
\psfrag{p0}{\tiny $_{\pi_{0}}$}
\psfrag{il}{\tiny $_{\iota_{\lambda}}$}
\psfrag{kl}{\tiny $_{\kappa_{\lambda}}$}
\psfrag{pl1}{\tiny $_{\pi_{\lambda}(1)}$}
\psfrag{rl1}{\tiny $_{\rho_{\lambda}(1)}$}
\psfrag{20m0}{\tiny $_{(20m,0)}$}
\psfrag{6m3}{\tiny $_{(6m,3)}$}
\psfrag{6m0}{\tiny $_{(6m,0)}$}
\psfrag{30m30}{\tiny $_{(30m,3)_0}$}
\psfrag{30m3I}{\tiny $_{(30m,3)_I}$}
\psfrag{6m3I}{\tiny $_{(6m,3)_I}$}
\includegraphics[scale=.65]{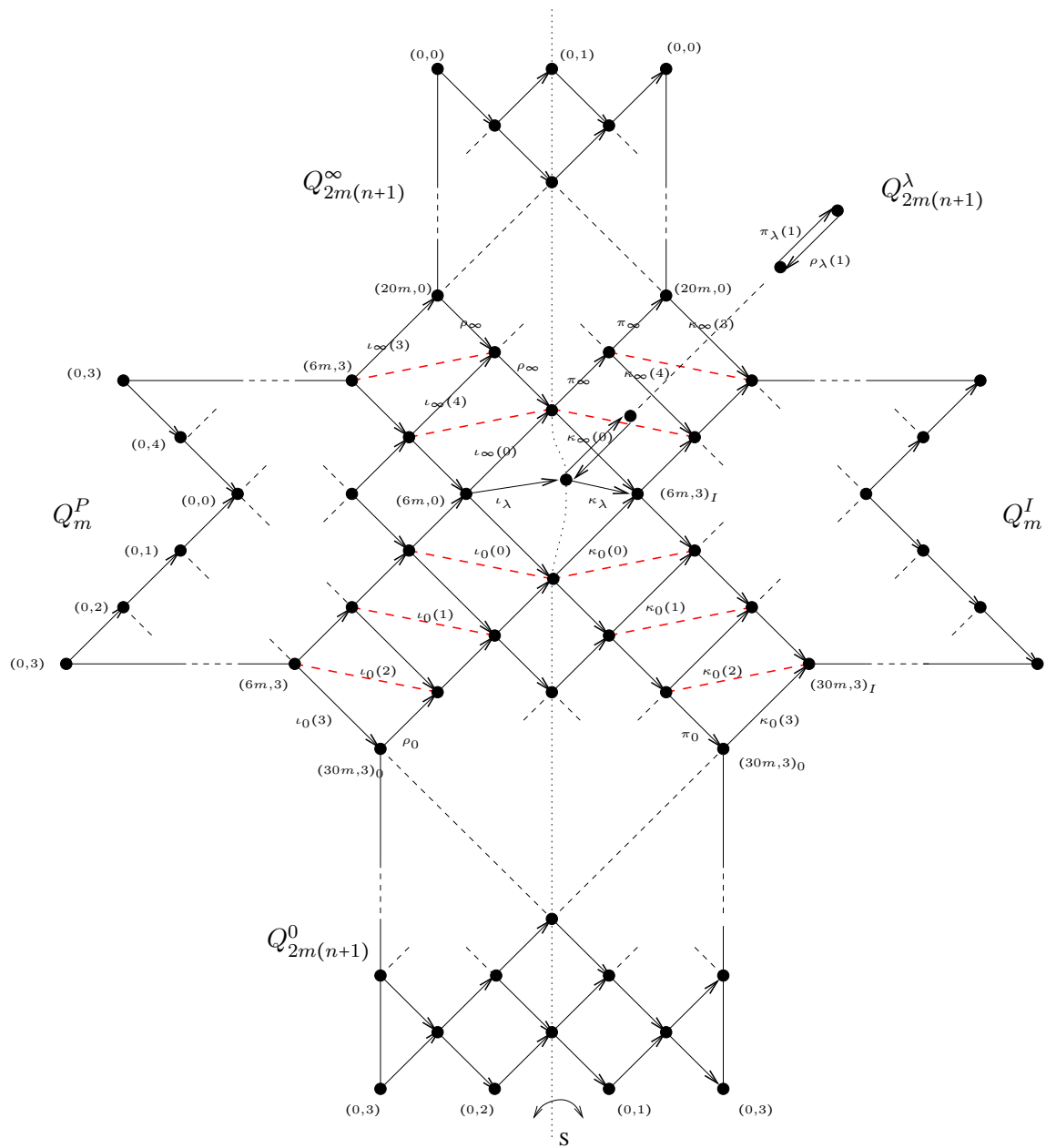}
\caption{$Q_m$ for $g=3$, $n=4$}
\label{fig:quiver-again}
\end{center}
\end{figure}

\end{document}